\UseRawInputEncoding
\documentclass[11pt, a4paper]{amsart}
\usepackage{amssymb}
\usepackage{amsopn}
\usepackage{amsmath}
\usepackage[colorlinks, citecolor=red]{hyperref}

\usepackage{multirow}
\usepackage[all]{xy}
\usepackage{amsfonts,amsthm,mathrsfs}

\numberwithin{equation}{section}
\newtheorem{theorem}[equation]{Theorem}
\newtheorem*{theorem*}{Theorem}
\newtheorem*{Hodge Conjecture}{Hodge Conjecture}
\newtheorem*{Hodge Conjecture (reduced)}{Hodge Conjecture (reduced)}
\newtheorem*{Lefschetz $(1,1)$-theorem}{Lefschetz $(1,1)$-theorem}
\newtheorem*{Jacobi Inversion Theorem}{Jacobi Inversion Theorem}
\newtheorem{lemma}[equation]{Lemma}

\newtheorem{proposition}[equation]{Proposition}
\newtheorem{conjecture}[equation]{Conjecture}

\theoremstyle{definition}

\theoremstyle{remark}
\newtheorem*{remark}{Remark}

\theoremstyle{plain}

\newcommand{\opname}[1]{\operatorname{\mathsf{#1}}}

\newcommand{\ie}{{\rm i.e.}}
\newcommand{\cf}{{\rm cf. }}

\newcommand{\C}{\mathbb{C}}

\newcommand{\Q}{\mathbb{Q}}
\newcommand{\Z}{\mathbb{Z}}
\renewcommand{\L}{\mathbb{L}}
\renewcommand{\P}{\mathbb{P}}

\newcommand{\pr}{\mathit{pr}}

\newcommand{\Hom}{\operatorname{Hom}}

\DeclareMathOperator{\im}{im}

\DeclareMathOperator{\coker}{coker}

\renewcommand{\Im}{\operatorname{Im}}

\DeclareMathOperator{\GL}{GL}

\newcommand{\N}{\mathbb{N}}
\newcommand{\sg}{\opname{sg}}
\newcommand{\sm}{\opname{sm}}
\newcommand{\F}{\mathbb{F}}
\newcommand{\G}{\mathbb{G}}
\renewcommand{\O}{\mathbb{O}}

\renewcommand{\H}{\mathcal{H}}
\renewcommand{\tilde}[1]{\widetilde{#1}}

\newcommand{\Aut}{\opname{Aut}}

\renewcommand{\a}{\opname{a}}
\newcommand{\Gr}{\opname{G}}
\newcommand{\U}{\mathbb{U}}
\newcommand{\I}{\mathfrak{I}}

\renewcommand{\pr}{\opname{pr}}

\newcommand{\Div}{\opname{Div}}

\newcommand{\prim}{\opname{prim}}
\newcommand{\van}{\opname{van}}
\newcommand{\p}{\opname{p}}
\newcommand{\PP}{\mathcal{P}}

\begin{document}


\title[A new approach towards Lefschetz $(1,1)$--Theorem]{A new approach towards Lefschetz $(1,1)$--Theorem}
\author{Erjuan Fu}
\address{Yau Mathematical Sciences Center,
Tsinghua University, Beijing, 100084, China}
 \email{ejfu@tsinghua.edu.cn; 473051732@qq.com}

\begin{abstract}
Let $S$ be a complex projective surface.
Lefschetz originally proved
Lefschetz $(1, 1)$--Theorem by studying a Lefschetz pencil of hyperplane sections of $S$ and the Abel--Jacobi mapping.
In this paper, we attack Lefschetz $(1, 1)$--Theorem by constructing certain two-parameter
families of twice hyperplane sections of $S$  and then applying the topological Abel--Jacobi mapping.
Our geometric constructions would give an inductive approach and some insight for higher dimensional cases.

We prove a strong tube theorem which generalizes Schnell's tube theorem to integral homology groups for complex projective curves and then obtain a Jacobi-type inversion theorem. In the end, we give a geometric description for the deformation space of an elementary vanishing cycle over a generic net.
\end{abstract}

\hypersetup{
linkcolor=blue}
\maketitle
\tableofcontents


\section{Introduction}\label{s:introduction}

Lefschetz $(1,1)$-Theorem \cite{Lef} is well known as the Hodge conjecture on $(1,1)$--classes. 

\begin{Lefschetz $(1,1)$-theorem} Let $S$ be a complex projective surface.
 Every integral primitive cohomology class of type $(1,1)$ on $S$ is a linear combination with integral coefficients of the cohomology classes of complex algebraic curves on $S$.
 \end{Lefschetz $(1,1)$-theorem}

By choosing a very ample divisor, we can assume that $S$ is a nondegenerate complex projective surface of degree $d$ in $\P^N$.  A cohomology class on $S$ is called \textit{primitive} if it restricts to zero on all smooth hyperplane sections.

Lefschetz originally proved this theorem by studying  a certain one-parameter
family of hyperplane sections of $S$ in $\P^N$, which is famous as a Lefschetz pencil \cite[Chapter 2]{Voisin2}.
In fact,  monodromy representations can be computed using a Lefschetz pencil 
by Zariski's Theorem \cite[Theorem 3.22]{Voisin2}.  Therefore, instead of considering all hyperplane sections, we only need to pick a Lefschetz pencil of hyperplane sections.

Besides the Lefschetz pencil, the other important tool in Lefschetz original method is the Abel--Jacobi mapping.
 Based on 
 a theorem of Zucker \cite{Zucker-NF} which says that there is a one-to-one correspondence between cohomology classes of type $(p,p)$ and normal functions, 
 Lefschetz's method would provide a sort of inductive proof of the Hodge conjecture if Jacobi Inversion Theorem holds, \ie, the Abel--Jacobi mapping is surjective.  
 Unfortunately, the  Jacobi Inversion Theorem fails in higher dimensions. Actually, the image of the Abel--Jacobi mapping for a  complex $(2n-1)$-dimension manifold $X$ is a complex subtorus of the  intermediate Jacobian $J(X)$ whose tangent space at $0$ is contained in $H^{n-1,n}(X)$ \cite[Corollary 12.19]{Voisin1}.
 That is to say, in general, it is impossible to parametrize the intermediate Jacobian $J(X)$ by \emph{algebraic cycles} of middle degree.

In recent years,  it has come to be understood that one-parameter families are not sufficient.  Instead, we construct certain two-parameter families called \emph{generic nets}.  The idea is to continue to take a Lefschetz fiberation for each hyperplane section over the first Lefschetz pencil  and make sure that any singular hyperplane section in the second fiberation has either at most two ordinary double points as singularities or only one ordinary cusp point as singularity. 

More precisely, denote by  $\O=:(\P^N)^\vee$ which parametrizes all hyperplane sections of $S$.
Pick a general line $\L$ in $\O$, 
then $\{S_t\}_{t\in \L}$ forms a Lefschetz pencil \cite[Corollary 2.10]{Voisin2}, where $S_t=\{x\in S\,|\, t(x)=0\}$.
In order to construct a generic net, we consider  the quotient bundle  \begin{equation*}\label{quot-bdle}
\mathcal{Q}:=\frac{{\O\times V}}{\mathcal{T}}
\end{equation*}
over $\O$, where $\O=\P(V)$ for some complex vector space $V\cong \C^{N+1}$ and
$\mathcal{T}$ denotes
the tautological line bundle 
over $\O$  whose fiber at each point is the line in $V$ represented by the point. Clearly, the fiber of the corresponding projective bundle $\P\mathcal{Q}$ over $t\in \O$ parametrizes all hyperplane sections of $S_t$.
Denote the pullback bundle over $\L$ by $\mathcal{Q}_\L=\L\otimes_\O \mathcal{Q}$.
 Since the first Chern class of $\mathcal{Q}$ is $1$, we have
\begin{equation*}\label{quot-bdle}
\mathcal{Q}_\L\cong \mathcal{O}_{\L}(1)\oplus \mathcal{O}_\L^{\oplus N-1}.
\end{equation*}
We then pick a general projective subbundle $\PP$  of rank $1$ over $\L$ in $\P\mathcal{Q}_\L$  such that
\begin{equation*}\label{quot-bdle}
\PP\cong \P(\mathcal{O}_{\L}(1)\oplus \mathcal{O}_\L^{\oplus 2})
\end{equation*}
and the fiber over each $t\in \L^{\sm}$
\begin{equation*}\label{quot-bdle}
\PP_t\cong \P(\mathcal{O}_{\L}(1)\oplus \mathcal{O}_\L)
\end{equation*}
 is a Lefschetz pencil for $S_t$, where  $\L^{\sm}$ denotes the open subset of $\L$ corresponding to all smooth hyperplane sections.

Any point in $\PP$ can be written as $(t,r)$ with $t\in \L$. Denote the twice hyperplane section over $(t,r)$ by $$S_{t,r}:=\{x\in S\,\vert\, t(x)=r(x)=0\}$$ which is a set of at most $d$ points in $S$.
Let $\PP^{\sm}$ be the open subset of $\PP$ such that all corresponding twice hyperplane sections  are $d$ distinct points.
We  call such $\PP$ as a \emph{generic net} if the complement $\PP\setminus \PP^{\sm}$
is an algebraic curve with only finite many ordinary nodes and ordinary cusps as its singularities. We have the following commutative diagram.
$$\xymatrix{
S_{\PP^{\sm}}:=S_{\L^{\sm}}{\times_{\L^{\sm}}} \PP^{\sm} \ar[d]^{\sigma} \ar[r]  & S_{\L^{\sm}} \ar[d]_{\pi}\ar[r]         &S  \\
\PP^{\sm}  \ar[r]^{\rho}  & \L^{\sm}     &           }$$
where $S_{\L^{\sm}}=\left\{(x, t)\in S\times \L^{\sm}\,|\, t(x)=0\right\}$  
denotes the incidence variety.

Next,  we use the topological Abel--Jacobi mapping \cite{TopAJ} to replace the Abel--Jacobi mapping so that we have a similar Jacobi inversion theorem. 
Denote by $H_{\prim}^{2}(S, \Q)$ the second primitive cohomology group of $S$. Nori¡¯s famous Connectivity
Theorem \cite[Corollary 4.4 on p. 364]{Nori} gives us the following isomorphism
\begin{align*}
H_{\prim}^{2}(S, \Q)&\cong H^1(\L^{\sm}, R^{1}_{\van}\pi^{\sm}_*\Q).
\\
    H^{1}(S_t, \Q)&\cong H^1(\PP_t^{\sm}, R^{0}_{\van}\sigma^{\sm}_{*}\Q),\qquad t\in \L^{\sm},
    \end{align*}
   where $R^{1}_{\van}\pi^{\sm}_*\Q$ denotes the local systems over $\mathcal{P}^{\sm}$ fiberwisely given by $H_{\van}^{1}(S_t, \Q)$ which is a summand of  $H^{1}(S_t, \Q)$ \cite[Proposition 2.27]{Voisin2} and $R_{\van}^0\sigma_*^{\sm}(\Z)$ denotes the local systems over $\mathcal{P}^{\sm}$ fiberwisely given by $$H_{\van}^0(S_{t,r}, \Z)=\bigoplus\limits_{i=1}^{d-1} \Z(P_i-P_0)$$ with $S_{t,r}:=\{P_0, P_1, \cdots, P_{d-1}\}$.
These isomorphisms lead us to consider $H_{\van}^{1}(S_t, \Z)$ (see Section \ref{s:prim-van} for basic definitions) and
the local system $R_{\van}^0\sigma_*^{\sm}(\Z)$ over $\mathcal{P}^{\sm}$.
 Via the flat Gauss-Manin connection, the total space naturally forms an infinite-sheeted covering space
$$|R_{\van}^0\sigma_*(\Z)|\to \mathcal{P}^{\sm}.$$
Pick an elementary vanishing cycle $\alpha\in H_{\van}^0(S_{t,r}, \Z)$.
By an elementary vanishing cycle, we mean that it specializes an indivisible vanishing cycle which vanishes as approaching to one singular fiber in a family. For instance, $P_i-P_j$ is an elementary vanishing cycle for any $i\neq j$.
There is a unique component $\mathbb{U}$ in $|R_{\van}^0\sigma_*^{\sm}(\Z)|$ corresponding to the stabilizer of $\alpha$ in $\pi_1(\mathcal{P}^{\sm})$. For any $t\in \L^{\sm}$, by composing the topological Abel--Jacobi mapping  \cite{TopAJ} with the projection $J(S_t) \to J_{\van}(S_t)$, we consider
$$\a_t: \mathbb{U}_t \to J_{\van}(S_t),$$
where  the vanishing intermediate Jacobian  $J_{\van}(S_t)$ is a quotient torus of the Jacobian variety $J(S_t)$  (see Section \ref{sec-top} for precise definitions).
We will use the fiber $\mathbb{U}_t$ to construct a parametrization of $J_{\van}(S_t)$ for all $t\in \L^{\sm}$.
%
%
In particular, we have the following similar Jacobi inversion theorem.

\begin{theorem}[Jacobi-type Inversion Theorem]\label{thm-Jacobi}
Let $S$ be a complex projective surface and $\PP$ be a generic net over a Lefschetz pencil $\L$ on $S$. Let $t\in \L^{\sm}$. There is a positive integer $K$ such that 
$$\a_t^{(k)}: \mathbb{U}_t^{(k)}:=\mathbb{U}_t\times_{\mathcal{P}_t}\mathbb{U}_t\times_{\mathcal{P}_t}\cdots \times_{\mathcal{P}_t}\mathbb{U}_t\to J_{\van}(S_t).$$ is surjective for all $k\geq K$.
\end{theorem}

Note that $J_{\van}(S_t)$ is a compact complex torus, we  prove this theorem by considering the induced morphism on fundamental groups, which coincides with the tube mapping constructed by Schnell \cite{Tube} up to a sign.
Inspired by Schnell's result \cite[Theorem 1]{Tube} and
note that the tube class of an elementary vanishing cycle is in $H_{1}^{\van}(S_t, \Z)$,
we prove Theorem \ref{thm-Jacobi} by showing the following strong tube theorem in advance.
\begin{theorem}[Strong Tube Theorem]\label{thm-tube}
Let $S$ be a complex projective surface and $\PP$ be a generic net over a Lefschetz pencil $\L$ on $S$. Let $(t,r)\in \PP^{\sm}$ with
$t\in \L^{\sm}$ and
 $\alpha\in H_{0}^{\van}(S_{t,r}, \Z)$ be an elementary vanishing cycle. Denote by $G_\alpha$  the stabilizer of $\alpha$ in $\pi_1(\mathcal{P}_t^{\sm}, r)$.
 Then the image of the topological Abel--Jacobi homomorphism
 $$\a_{t*}: G_{\alpha} \to H_{1}^{\van}(S_{t}, \Z),$$
 is $m H_{1}^{\van}(S_t, \Z)$ for some positive integer $m$.
\end{theorem}


We conjecture this result holds in higher dimensional cases by connecting with Schnell's tube mapping in Section \ref{sec-top}.

Now, we go back to the Lefschetz $(1,1)$--Theorem. 
Given a nonzero primitive cohomology class $\eta$ of type $(1,1)$, there is a normal function
$\nu_\eta\in H^0(\L^{\sm}, \mathcal{J})$ such that $\eta=[\nu_\eta]$ is the cohomology class of $\nu_\eta$, where $\mathcal{J}:=\bigcup\limits_{t\in\L^{\sm}} J(S_t)$ is a complex manifold and forms an analytic fiber space.
With the help of Theorem \ref{thm-Jacobi}, we can lift the normal function $\nu_\eta$ locally continuously to a section
$\mu_\eta: \L^{\sm} \to \mathbb{U}^{(N)}$
such that we have the following commutative diagram
$$\xymatrix{\mathbb{U}^{(N)}\ar[r]^{\a^{(N)}}\ar@{<-}[rd]^{\mu_\eta}&\mathcal{J}\ar@{<-}[d]^{\nu_\eta}\\ \PP^{\sm}\ar@{.>}[r]_\rho\ar@{<.}[u]^\sigma&\L^{\sm}}.$$

Following the idea in \cite{HodgeLocus}, we construct a canonical completion $\overline{\mathbb{U}}$ of $\mathbb{U}$ to a normal analytic variety and with a proper morphism
$$\overline{\mathbb{U}}\xrightarrow{\overline\sigma} \mathcal{P}$$
which extends $\mathbb{U}\xrightarrow{\sigma} \mathcal{P}^{\sm}$. 
We study the geometry of this completion.
Specifically,
let $\widehat{\mathcal{P}}$ be the blowup of $\mathcal{P}$ along $\Sigma:=\mathcal{P}\setminus \mathcal{P}^{\sm}$ and $\widehat\Sigma$ be the proper transform of $\Sigma$ in $\widehat{\mathcal{P}}$. Denote by $\tilde{\mathcal{P}}$ the double cover of $\widehat{\mathcal{P}}$ branched along $\widehat\Sigma$. In fact, $\tilde{\mathcal{P}}$ is smooth and is the minimal resolution of the double cover of ${\mathcal{P}}$ branched along $\Sigma$. We have the following conclusion about the geometry of $\overline{\mathbb{U}}$.

\begin{theorem}\label{thm-geo}
Let $S$ be a complex projective surface and $\PP$ be a generic net over a Lefschetz pencil $\L$ on $S$.
 We have
$$\overline{\mathbb{U}} \cong S_{\widetilde{\mathcal{P}}}\times_{\widetilde{\mathcal{P}}}S_{\widetilde{\mathcal{P}}}\setminus\Delta,$$
where  $S_{\widetilde{\mathcal{P}}}=\widetilde{\mathcal{P}}\times_{\mathcal{P}^{\sm}}S_{{\mathcal{P}^{\sm}}}$ and $\Delta$ denotes the diagonal of the fiber product.
\end{theorem}
The geometry of the deformation space $\overline{\mathbb{U}}$ will help us to search the algebraic 0--cycles on $S_t$, which will trace out an algebraic 1--cycle dual to $\eta$ as $t$ moves over $\L$.
In fact, the closure of $\mu_\eta(\L^{\sm})$ in $\overline{\mathbb{U}^{(N)}}$ is an algebraic 1--cycle, which is desired.

\subsection*{Notations.}Throughout this paper, the integral (co)homology means the integral (co)homology \emph{modulo the torsion}, and we will use the following notations.
\begin{itemize}
\item $M$: a complex project manifold of dimension $m$;
\item $\O_M$: the projective space parametrizing all hyperplane sections of $M$. Sometimes, we use $\O$ for short if there is no confusing.
\item $\L_M$: a Lefschetz pencil of $M$ which is a projective line in $\O_M$. Sometimes, we use $\L$ for short if there is no confusing.
\item $S$: a complex project manifold of dimension $2n$. In particular, $S$ is a complex projective surface if $n=1$.
    \item $X$: a complex project manifold of dimension $2n-1$.  In particular, $X$ is a complex projective curve if $n=1$.
    \item $A_B$: the incidence variety defined by
    $$A_B:=\left\{(a,b)\in A\times B\,\vert \, b(a)=0\right\},$$
    where $B$ parametrizes the hyperplane sections of $A$. And the two natural projection denote by
    $\p_A: A_B \to A$ and $\p_B: A_B\to B$.
    \item $\P(V)$: the projectivization of $V$, \ie, the set of all $1$--dimensional vector subspaces in $V$, where $V$ is a complex vector space.
    \item $\P(V)^\vee:=\P(V^\vee)$, where $V^\vee:=\Hom_{\C}(V, \C)$ denotes the dual vector space of $V$.
\end{itemize}

\subsection*{Organization of this paper}
In section \ref{sec-pre}, we review the Hodge conjecture and Lefschetz's original proof for the case of complex projective surface.
In Section \ref{sec-top}, we first review the topological Abel--Jacobi mapping and the tube mapping, then we explain the connection between these two mappings. Furthermore, we will give the conjectures towards the strong tube theorem and Jacobi--type inversion theorem in higher dimensional cases inspired by Schnell's result \cite[Theorem 1]{Tube}.
We prove the strong tube theorem in Section \ref{sec-tube} and the Jacobi--type inversion theorem in Section \ref{sec-Jacobi} for curve case. In Section \ref{sec-geo}, we give a geometriy construction for the deformation space $\overline{\U}$ and prove Theorem \ref{thm-geo}.

\subsection*{Acknowledgments}
This work is part of  my Ph.D. thesis.
I am deeply grateful to my advisor Herb Clemens for introducing me this beautiful geometric construction and his invaluable intuition which have provided ideas for many proofs.
 I also would like to express my truly gratitude to my advisor Yuan-Pin Lee for his support and many useful advice. 
 Besides, I would like to thank Yilong Zhang for some useful communications.


%

\section{Preliminaries}\label{sec-pre}

In this section, we will establish some background.
The familiar reader can skip this section and use it as a reference.

Let $M$ be a complex projective manifold of dimension $m$. 
The Fubini-Study metric induced on $M$ is a K\"{a}hler metric, thus we have the Hodge Decomposition \cite{Voisin1}
\begin{equation*}
H^k(M, \C)=\bigoplus\limits_{p=0}^k H^{p, k-p}(M) , \qquad \overline{H^{p, q}}=H^{q, p}.
\end{equation*}

By a Hodge class of degree $2m$, we mean a  cohomology class  of type $(m,m)$.
The modern statement of the Hodge conjecture is as follows.
\begin{Hodge Conjecture}
Every  rational Hodge class of degree $2m$  on $M$ is algebraic, \ie, it is a linear combination with rational coefficients of the cohomology classes of complex  subvarieties of codimension $m$. 
\end{Hodge Conjecture}
Note that the Hodge conjecture can be reduced to \emph{primitive} Hodge classes of middle degree
\cite[Lecture 14]{Lewis}.
By a primitive $k$-th cohomology class, we mean that it is killed by $L^{n-k+1}$, where $L$ denotes the Lefschetz operator, \ie, represents cup product with the K\"{a}hler class on $M$. The primitive and vanishing (co)homologies are recalled here.

\subsection{Primitive and vanishing (co)homology}\label{s:prim-van}
By choosing a very ample divisor, we can assume that $M$ is nondegenerate    in $\P^{\overline{m}}$ for some positive integer $\overline{m}$.
Let $M_0$ be a smooth hyperplane section of $M$ in $\P^{\overline{m}}$.
 Denote by $i: M_0 \hookrightarrow M$ the including.
The Lefschetz Hyperplane theorem tells us that 

\begin{enumerate}
    \item
 $H_k(M_0,\Z) \xrightarrow{i_*} H_k(M, \Z)$ is an isomorphism for $k<m-1$ and is surjective for $k=m-1$;
 \item $H^k(M,\Z) \xrightarrow{i^*} H^k(M_0, \Z)$ is an isomorphism for $k<m-1$ and is injective for $k=m-1$.
\end{enumerate}
Moreover, by Poincar\'e Duality, we have
 $$H_k(M,\Z) \cong H_{2m-k}(M,\Z)^\vee\qquad \text{and}\qquad
 H^k(M,\Z) \cong H^{2m-k}(M,\Z)^\vee.$$ Therefore,
the $k$-th (co)homology of $M$ can be inferred from $M_0$ for all $k\neq m$, and the only piece of the $m$-th  (co)homology that cannot be inferred is the following: 
\begin{enumerate}
\item $H_m^{\prim}(M,\Z):=H_m(M,\Z)/i_*H_m(M_0,\Z)=\coker\left\{H_m(M_0,\Z) \xrightarrow{i_*} H_m(M, \Z)\right\}$,
\item $H^{m}_{\opname{prim}}(M,\Z):=\ker\left\{H^{m}(M,\Z) \xrightarrow{i^*} H^{m}(M_0, \Z)\right\} \cong
\ker\left\{H_{m}(M,\Z) \xrightarrow{i^*} H_{m-2}(M_0, \Z)\right\} $,
\end{enumerate}
which is called the $m$--th \emph{primitive} (co)homology of $M$.
Since the dual functor $\Hom(\cdot, \Z)$ is left exact,  we have
 $$H^m_{\prim}(M, \Z)\cong \Hom(H_m^{\prim}(M,\Z), \Z):=H_m^{\prim}(M,\Z)^\vee.$$
 Consequently, we only need to study one of them, \ie, either $H^m_{\prim}(M, \Z)$ or $H_m^{\prim}(M, \Z)$.

Similarly, we have $(m-1)$--th primitive (co)homology for $M_0$ whose certain subgroup called the $(m-1)$--th vanishing (co)homology group,
turns out to be very important. The $(m-1)$--th vanishing (co)homology group of $M_0$ is defined as follows:
 \begin{align*}
 H_{m-1}^{\van}(M_0,\Z)&=\ker\left\{H_{m-1}(M_0,\Z) \xrightarrow{i_*} H_{m-1}(M, \Z)\right\},\\
  H^{m-1}_{\van}(M_0,\Z)&=\ker\left\{H^{m-1}(M_0,\Z) \xrightarrow[\text{Gysin}]{i_*} H^{m+1}(M, \Z)\right\}\cong H_{m-1}^{\van}(M_0, \Z).
 \end{align*}
Furthermore, we have the following decompositions as orthogonal direct sums related to the intersection form \cite[Proposition 2.27]{Voisin2}:
 \begin{align}
 H^{m-1}(M_0,\Q)&=H_{\van}^{m-1}(M_0,\Q) \oplus i^*H^{m-1}(M, \Q), \label{eqn-summand-1}\\
  H^{m-1}_{\prim}(M_0,\Q)&=H^{m-1}_{\van}(M_0,\Q) \oplus i^*H^{m-1}_{\prim}(M, \Q) \label{eqn-summand-2}.
 \end{align}
 For hypersurfaces in projective space, the vanishing middle degree (co)homology group is the same as the primitive middle degree (co)homology group by the second decomposition in the above (\cf \cite[Remark 6.8 on p.163]{Voisin2}).
\begin{remark}
Note that the vanishing and primitive (co)homology are self-dual under the intersection pairing, we will exchange sometimes between homology and cohomology for convenience in this paper.
\end{remark}

With the help of the primitive (co)homology classes, the Hodge conjecture can be restated as follows.

\begin{Hodge Conjecture (reduced)}\label{redHC}
Let $S$ be a complex projective manifold of \emph{even} dimension $2n$, then every primitive Hodge class of degree $2n$ is algebraic.
\end{Hodge Conjecture (reduced)}


Up to now, the Hodge conjecture was only proved for a few special cases \cite{Lewis, Zucker-4folds}, for example, most abelian varieties, some hypersurfaces and the case of divisors, \ie, $n=1$. In particular, for $n=1$ case, the Hodge conjecture is true for all  \emph{integral} cohomology classes, which is known as the famous Lefschetz $(1,1)$-theorem.

Lefschetz originally proved the Lefschetz $(1,1)$-theorem  by taking a Lefschetz pencil of hyperplane sections and then using Poincar\'e normal function and the Abel--Jacobi mapping.
 We review these concepts in the following.

\subsection{Lefschetz pencils \cite[Chapter 2]{Voisin2}}\label{s:LefBlowup}
 Note that $M$ is nondegnerate in $\P^{\overline{m}}$, then $\O_M:=(\P^{\overline{m}})^\vee$ parametrizes all hyperplane sections of $M$ in $\P^{\overline{m}}$. Any projective line $\P^1$ in $\O_M$ parametrizes a pencil of hyperplane sections $\{M_t\}_{t\in \P^1}$  of  $M$, which is called a \emph{Lefschetz pencil} if it satisfies the following two conditions:
\begin{enumerate}
\item The base locus $B:=\bigcap_{t\in \P^1} M_t$ is smooth of codimension 2 in $M$.
\item Every hypersurface $M_t$ has at most one ordinary double point as singularity.
\end{enumerate}

On the other hand,
let  $Z\subseteq M$ be a complex submanifold of codimension $k$. Locally along $Z$, there exist holomorphic functions $f_1, \cdots, f_k$ with independent differentials such that $Z \cap U=\{ z \in U | f_i(z)=0, \, i=1, \cdots, k\}$, where $U$ is an open subset of $M$.  Set
\begin{equation*}
\tilde U_Z:=\left\{\left([X_1, \cdots, X_k], z\right)\in \P^{k-1}\times U \,|\,
\opname{rank}\left(\begin{array}{cccc} X_1&X_2&\cdots &X_k\\
f_1(z)&f_2(z)&\cdots &f_k(z)\end{array}\right)=1\right\},
\end{equation*}
which is the blowup of $U$ along $Z\cap U$.  We can glue $\tilde U_Z$ together for all open subsets $U\subseteq M$ to get the blowup of $M$ along $Z$, which we denote by $\opname{Bl}_ZM$.

Suppose $\{M_t\}_{t\in \L}$ is a Lefschetz pencil with $\L\subseteq \O_M$ being a projective line.  Locally, in an open subset $U$ of $M$,  there exist holomorphic functions $f$ and $g$ with independent differentials such that $M_t \cap U=\{ z \in U | t_1f(z)+t_2g(z)=0\}$, where $t=[t_1, t_2]\in \L$.  Note that
$B\cap U=\{ z \in U | -g(z)=f(z)=0\}$,  then
\begin{equation*}
\begin{aligned}
\tilde U_B:=&\left\{\left([t_1, t_2], z\right)\in \L\times U \,|\,
t_1f(z)+t_2g(z)=0\right\}\\
=&\left\{\left([t_1, t_2], z\right)\in \L\times U \,|\,
z\in M_{[t_1,t_2]}\right\}.
\end{aligned}
\end{equation*}
Glue $\tilde U_B$ together for all open subsets $U\subseteq M$,  we get the blowup $\opname{Bl}_B M$. Note that the incidence variety of $M$ along the Lefschetz pencil $\{M_t\}_{t\in \L}$ is defined by
$$M_\L:=\left\{\left([t_1, t_2], z\right)\in \L\times M \,|\,
z\in M_{[t_1,t_2]}\right\}.$$
Clearly, we have $M_\L=\opname{Bl}_B M$.

 \subsection{The Abel--Jacobi mapping}
Let $X$ be an irreducible smooth complex projective algebraic curve  of genus $g\geq 1$, and the associated analytic variety is a connected closed Riemann surface.
Let $\Omega$ be the bundle of holomorphic $1$-form on $X$. The Jacobian of $X$  is defined by
$$J(X)=\frac{H^0(X, \Omega)^\vee}{H_1(X, \Z)}
$$
which is a complex torus and isomorphic to $\C^g/\Z^{2g}$, and
$$H_1(X, \Z) \hookrightarrow H^0(X, \Omega)^\vee$$
is given by integrating forms over chains.

 Fix a base point $p_0\in X$, the Abel--Jacobi mapping is defined as the following:
 \begin{equation*}
\begin{aligned}
\a_X : X &\to J(X)\\ 
p  &\mapsto \int_{p_0}^p. 
\end{aligned}
\end{equation*}
For any positive integer $d$, denote by $\opname{Sym}^dX$ the $d$-th symmetric product of $X$ that is a compact complex manifold of dimension $d$. The Abel--Jacobi mapping $\a_X$ can be naturally extended to $\opname{Sym}^dX$, and we still denote by $\a_X$:
\begin{equation*}
\begin{aligned}
\a_X : \opname{Sym}^dX &\to J(X)  \\
\{p_i\}_{i=1}^d  &\mapsto \sum\limits_{i=1}^d\a_X({p_i})=\sum\limits_{i=1}^d\int_{p_0}^{p_i}.
\end{aligned}
\end{equation*}
There is a theorem of the surjectivity about $\a_X$  \cite[p. 235]{GH-Principles}.

\begin{Jacobi Inversion Theorem}
 If $d\geq g$, the Abel--Jacobi mapping $\a_X : \opname{Sym}^dX \to J(X)$ is surjective. In particular,
$\a_X : \opname{Sym}^gX \to J(X)$ is a birational map.
\end{Jacobi Inversion Theorem}

Denote by $\opname{Div}(X)$ the set of all divisors on $X$, which is an abelian group and  $\opname{Div}^0(X)$ is the subgroup  of all divisors of degree 0. The Abel-Jacobi mapping $\a_X$ can be linearly extended to $\Div(X)$, and we still denote by $\a_X$:
\begin{equation*}
\begin{aligned}
\a_X : \Div(X) &\to J(X)  \\
\sum\limits_{i=1}^m n_i p_i  &\mapsto \sum\limits_{i=1}^mn_i\a_X({p_i})=\sum\limits_{i=1}^mn_i\int_{p_0}^{p_i}.
\end{aligned}
\end{equation*}
Note that $\a_X(p-q)=\int_{p_0}^p-\int_{p_0}^q=\int_q^p$, then $\a_X|_{\Div^0(X)}$ can be defined without the choice of the base point, and we still denote by $\a_X$:
\begin{equation*}
\begin{aligned}
\a_X : \Div^0(X) &\to J(X)  \\
\sum\limits_{i=1}^m n_i (p_i - q_i)&\mapsto \sum\limits_{i=1}^mn_i\a_X({p_i}-q_i)=\sum\limits_{i=1}^mn_i\int_{q_i}^{p_i}.
\end{aligned}
\end{equation*}

More generally, we still have an Abel--Jacobi mapping if the dimension of $X$ is $2n-1$.

\subsection{Generalized Abel--Jacobi mapping and Normal functions}\label{NF-HC} 

Let $X$ be a  complex projective
manifold of dimension $2n-1$. Denote by $F^\bullet$ the Hodge filtration on $H^{2n-1}(X,\C)$, \ie, $$F^kH^{2n-1}(X, \C)=\bigoplus\limits_{p\geq k}H^{p,2n-1-p}(X).$$



The intermediate Jacobian $J(X)$ of $X$ is defined as
\begin{equation*}
J(X)=\frac{(F^nH^{2n-1}(X,\C))^\vee}{H_{2n-1}(X, \Z)}\cong\frac{H^{2n-1}(X,\C)}{F^nH^{2n-1}(X)\oplus H^{2n-1}(X, \Z)},
\end{equation*}
which is a complex torus. Clearly, the intermediate Jacobian is the usual Jacobian variety of a compact Riemann surface for $n=1$.
Let $\Theta^n(X)$ be the group of algebraic cycles of codimension $n$ on $X$, \ie, the group of combinations $\sum_{i}n_iZ_i$, where the $n_i$ are integers and only finitely many of them are nonzero, and each $Z_i\subseteq X$ is a closed, irreducible, and reduced algebraic subset of codimension $n$ of $X$.
Let $\Theta^n(X)_{\hom}$ be the subgroup of algebraic cycles in $\Theta^n(X)$ that are homologically equivalent to $0$. The generalized Abel--Jacobi mapping for $X$ is defined by Griffiths as  the  following \cite[Section 12.1.2 on p. 292]{Voisin1}:
\begin{equation}\label{Griffiths-AJ}
\begin{aligned}
\a_X : \Theta^n(X)_{\hom} \to J(X)\\
Z \mapsto a_X(Z)=\int_\Gamma,
\end{aligned}
\end{equation}
where $\Gamma \subseteq X$ is a differentiable chain of dimension $n$ such that $\partial \Gamma =Z$.

Let $\pi : \mathcal{X} \to B$ be a holomorphic projective fibration of relative dimension $2n-1$ between two complex manifolds, which induces the intermediate Jacobian fibration $J\to B$ with fiber $J(X_b)$ at $b\in B$. Denote by $\mathcal{J}$ the sheaf of holomorphic sections of $J\to B$.
A normal function means a holomorphic section $\nu$ of $J\to B$, \ie,
$\nu \in H^0(B, \mathcal{J})$.

Let $\Theta^n(X/B)_{\hom}$ be the subgroup of relative algebraic cycles in $\Theta^n(X/B)$ that are homologically equivalent to $0$ along the fiber, \ie, $Z_b:=\mathcal{Z}\cap X_b$ are homologous to $0$ for all $b\in B$.
For any $\mathcal{Z} \in \Theta^n(X/B)_{\hom}$, we have a horizontal normal function $\nu_{\mathcal{Z}}\in H^0(B, \mathcal{J})$ defined as the following \cite[Section 7.2.1 on p. 193]{Voisin2}:
\begin{equation*}
\begin{aligned}
\nu_{\mathcal{Z}}: &B\to J \\
&b \mapsto \a_{X_b}(Z_b) \in J(X_b).
\end{aligned}
\end{equation*}


Denote by $H_\Z^{2n-1}: = \mathrm{R}^{2n-1} \pi_* \Z$ the local system over $B$ and $\H^{2n-1}:=H_\Z^{2n-1}\otimes_\Z \mathcal{O}_{B} $.
By the Hodge decomposition, there is an integral variation of Hodge structures $(H^{2n-1}_\Z, F^\bullet \H^{2n-1})$ of weight ${2n-1}$, then $\mathcal{J}$ can also be obtained by the following short exact sequence
\begin{equation*}
0 \to H^{2n-1}_\Z \to \H^{2n-1}/F^n\H^{2n-1} \to \mathcal{J} \to 0,
\end{equation*}
which induces a long exact sequence in cohomology
\begin{equation*}
 H^0(B, \mathcal{J}) \xrightarrow{\alpha} H^1(B, H^{2n-1}_\Z)\xrightarrow{\beta} H^1(B, \H^{2n-1}/F^n\H^{2n-1}).
\end{equation*}
For any $v \in  H^0(B, \mathcal{J})$, since $\beta(\alpha(v))=0$ in
$$H^1(B, \H^{2n-1}/F^n\H^{2n-1})\cong \frac{H^1(B, \mathrm{R}^{2n-1}\pi_*\C)}{F^nH^1(B, \mathrm{R}^{2n-1}\pi_*\C)},$$
we obtain that  $\alpha(v)$ is a Hodge class \cite[Section 8.2.2 on p. 229]{Voisin2}.


Let $L^\bullet$ be the Leray filtration for $\pi: \mathcal{X} \to B$, 
 then we have
\begin{equation}
L^1H^{2n}(\mathcal{X}, \Z)=\ker\left\{H^{2n}(\mathcal{X}, \Z)\to H^{2n}(X_b, \Z)\right\},
\end{equation}
which gives us $[\mathcal{Z}]\in L^1H^{2n}(\mathcal{X}, \Z)$ \cite[Section 8.2.2 on p. 230]{Voisin2}.
Moreover, we have a morphism of Hodge structures
\begin{equation}
\eta: L^1H^{2n}(\mathcal{X}, \Z) \to H^1(B, H^{2n-1}_\Z).
\end{equation}
Therefore, we have $\alpha(\nu_\mathcal{Z})=\eta([\mathcal{Z}])$ 
\cite[Lemma 8.20 on p. 230]{Voisin2}.

 \medskip

Now, we are ready to outline Lefschetz's orignial proof for the Lefschetz $(1,1)$--Theorem.

\subsection{Lefschetz's orignial proof}
Let $S$ be a nondegenerate complex projective surface of degree $d$ in $\P^N$. Denote by $\O=:(\P^N)^\vee$ which parametrizes all hyperplane sections of $S$.
Pick a general line $\L$ in $\O$, 
then $\{S_t\}_{t\in \L}$ forms a Lefschetz pencil. 
Denote by
 \begin{align*}
 \L^{\sg}=\{t_1, \cdots, t_m\}\subseteq \L &\text{ the subset corresponding to all singular hyperplane sections,}\\
 \L^{\sm}\subseteq \L &\text{ the subset corresponding to all smooth hyperplane sections}.
 \end{align*}

 For any $t\in \L^{\sm}$, $S_t$ is a smooth complex projective curve. We have the Jacobian variety denoted by $J(S_t)$, which is an abelian variety (connected compact complex Lie group). In a natural way, the Jacobi bundle $\mathcal{J}:=\bigcup\limits_{t\in\L^{\sm}} J(S_t)$ forms a complex manifold. Any holomorphic section of the Jacobi bundle is called a \emph{(Poincar\'e) normal function}.

 For  any $t_i \in \L^{\sg}$, denote by $\tilde{S}_{t_i}$ the normalization of $S_{t_i}$. We can define a generalized intermediate Jacobian $J(S_{t_i})$ as an extension of $J(\tilde{S}_{t_i})$ by $\C^*$, \ie, we have the following short exact sequence  \cite[Section 16]{RatInt}:
 $$ 1\to \C^* \to J(S_{t_i})\to J(\tilde{S}_{t_i})\to 0.$$
We conclude that the extended Jacobi bundle $\tilde{\mathcal{J}}:=\bigcup_{t\in\L} J(S_t)$ is a complex manifold that forms an analytic fiber space $\tilde{\mathcal{J}} \rightarrow \L$ of complex Lie groups  \cite[Theorem 17.1]{RatInt}.  Any holomorphic section of the extended Jacobi bundle is called an extended (Poincar\'e) normal function.

   Denote by
 \begin{equation*}
S_\L:=\left\{(x, t) \in S\times \L \,\big |\,  t(x)=0 \right\}=\opname{Bl}_B S,
\end{equation*}
where $B$ denotes the base locus of this Lefschetz pencil (Clearly, $\dim B=0$).
We have the following diagram
$$\xymatrix{&S_{\L}^{\sm}
\ar@{^{(}->}[r]\ar[d]_{\p_{\L}^{\sm}}
&S_\L\ar[r] \ar[d]_{\p_{\L}} &S\\ &   \L^{\sm}\ar@{^{(}->}[r]&\L}.$$

Denote by $H_\Z^i: = \mathrm{R}^i \p_{\L*}^{\sm}\Z$ 
 the local system over $\L^{\sm}$ and $\H^i:=H_\Z^i\otimes_{\Z} \mathcal{O}_{\L^{\sm}} $.
By the Hodge decomposition, we have an integral variation of Hodge structures $(H^1_\Z, F^\cdot \H^1)$ of weight $1$ over $\L^{\sm}$. By the definition of the Jacobi bundle, we have the following short exact sequence of sheaves over $\L^{\sm}$
\begin{equation*}
0 \to H^1_\Z \to \H^1/F^1\H^1 \to \mathcal{J} \to 0,
\end{equation*}
which induces a long exact sequence in cohomology
\begin{equation*}
H^0(\L^{\sm}, \mathcal{J}) \xrightarrow{\delta} H^1(\L^{\sm}, H^1_\Z) \to H^0(\L^{\sm}, \H^1/F^1\H^1).
\end{equation*}
For any $\eta \in \opname{Hdg}^1(S, \Z) = H^2(S, \Z) \cap H^{1,1}$, we have
$$\eta \in \ker\left\{H^1(\L^{\sm}, H^1_\Z) \to H^0(\L^{\sm}, \H^1/F^1\H^1)\right\},$$ so there is a normal function $\nu_\eta \in H^0(\L^{\sm}, \mathcal{J})$ such that $\delta(\nu_\eta)=\eta$.

  Let $g$ be the genus of $S_t$, which is independent of $t\in \L^{\sm}$.  Fix a base point $b\in B$,  consider the Abel--Jacobi mapping for each $t\in \L^{\sm}$
\begin{equation*}
\begin{aligned}
\a_t : \opname{Sym}^gS_t &\to J(S_t)  \\
p_1+\cdots+p_g  &\mapsto \sum\limits_{i=1}^g\a_t({p_i})=\sum\limits_{i=1}^g\int_{b}^{p_i},
\end{aligned}
\end{equation*}
which is surjective by the Jacobi Inversion Theorem \cite[p. 235]{GH-Principles}.
 Hence,
there are $g$ points $p_1^t, \cdots, p_g^t\in S_t$ for all $t\in \L^{\sm}$ such that
$$\a_t(p_1^t+ \cdots + p_g^t)=\nu_\eta(t).$$
Let $t$ continuously move in $\L^{\sm}$, then each $p_i^t$ traces out a curve in $S$. Since $\tilde{\mathcal{J}}$ is a compact complex manifold, taking the closure of these curves, we obtain $g$  algebraic cycles $\gamma_i$ in $S$. Then $\gamma:=\sum\limits_{i=1}^g \gamma_i$ is the algebraic cycle traced out by $\nu_\eta$.
Finally, since both $\eta$ and $\gamma$ represent the cohomology class of the normal function $\nu_\eta$, then $\eta=[\gamma]$ is algebraic,
which is an outline of Lefschetz's original proof.


\section{Topological Abel--Jacobi mapping and Tube mapping}\label{sec-top}

Let $S$ be a nondegenerate complex projective manifold of dimension $2n$ in $\P^{N}$ and $X$ be a smooth hyperplane section of $S$ in $\P^N$. Let $\O_X^{\sm}$ be the set of all $(N-2)$--planes $H$ such that $X\cap H$ is smooth.   Pick any $t\in \O_X^{\sm}$ and denote by $X_t$ the corresponding smooth hyperplane section. We first recall the definition of the topological Abel--Jacobi mapping for $X_t$.

For any $\alpha \in H_{2n-2}^{\van}(X_{t}, \Z)$, let  $\gamma_\alpha$ be a  vanishing $(2n-2)$-cycle on $X_{t}$ representing $\alpha$, \ie, $[\gamma_\alpha]=\alpha$.  There is a $(2n-1)$-chain $\Gamma_\alpha$ on $X$ such that $\partial \Gamma_\alpha=\gamma_\alpha$.  However, $\int_{\Gamma_\alpha}$ is not well-defined in $\left(F^nH^{2n-1}(X; \mathbb{C})\right)^\vee$ since $\int_{\Gamma_\alpha} d\varphi =\int_{\gamma_\alpha} \varphi$ may not be zero for a differential $(2n-2)$-form $\varphi$ of type $(n,n-2)+(n+1,n-3)+\cdots +(2n-2,0)$ \cite[Proposition 7.5]{Voisin1},  so we have to subtract this part from $\int_{\Gamma_\alpha}$.
In order to do this, we consider
the closed differential $(2n-2)$-form $\omega_\alpha$ on $X_{t}$ that is dual to $\gamma_\alpha$ in the sense that for all closed $(2n-2)$-forms $\varphi$
\begin{equation}\label{eqn-AJ-int-1}
\int_{\Gamma_\alpha} d\varphi =\int_{\gamma_\alpha} \varphi =\int_{X_t} \omega_\alpha\wedge \varphi =\int_{X} i_*\omega_\alpha\wedge \varphi,
\end{equation}
where $i_*$ is the Gysin morphism induced by $i: X_{t} \hookrightarrow X$.
Note that the cohomology class of $i_*\omega_\alpha$ is $0$, by Hodge decomposition on forms,  we have
\begin{equation}\label{eqn-AJ-int-2}
i_*\omega_\alpha=dd^*Gi_*\omega_\alpha,
\end{equation}
where $G$ is the Green's operator. Apply (\ref{eqn-AJ-int-2}) to (\ref{eqn-AJ-int-1}) gives us
\begin{equation}\label{eqn-AJ-int-3}
\int_{\Gamma_\alpha} d\varphi  =\int_{X} dd^*Gi_*\omega_\alpha \wedge \varphi= \int_{X} d^*Gi_*\omega_\alpha \wedge d\varphi.
\end{equation}
Consequently, we  define the topological Abel--Jacobi mapping as
 \begin{align*}
 \opname{a}_{X,t}^{\opname{top}}: H_{2n-2}^{\van}(X_{t}, \Z) &\to J_{\prim}(X),\\
\alpha &\mapsto \int_{\Gamma_\alpha}  -\int_{X}d^*Gi_*\omega_\alpha \wedge,
 \end{align*}
where
 the primitive intermediate Jacobian of $X$ is defined as
\begin{align*}
J_{\prim}(X)=   \frac{\left(F^nH^{2n-1}_{\prim}(X; \mathbb{C})\right)^\vee}{H_{2n-1}^{\prim}(X; \mathbb{Z})}\cong \frac{H^{2n-1}_{\prim}(X; \mathbb{C})}{F^nH^{2n-1}_{\prim}(X; \mathbb{C})\oplus H^{2n-1}_{\prim}(X; \mathbb{Z})},
\end{align*}
which is a quotient torus of $J(X)$.
Note that $\Gamma_\alpha \in  H_{2n-1}(X, X_{t}; \Z)$ and $\omega_\alpha\in H^{2n-2}_{\van}(X_{t}, \Z)$, and  we have the following two short exact sequences:
\begin{align*}
&0\to H_{2n-1}^{\prim}(X, \Z)\to H_{2n-1}(X, X_{t}; \Z) \to H_{2n-2}^{\van}(X_{t}, \Z) \to 0,\\
&0 \to H^{2n-1}_{\prim}(X, \Z) \to H^{2n-1}(X-X_{t}, \Z) \to H^{2n-2}_{\van}(X_{t}, \Z) \to 0,
\end{align*}
then the different choice of $\Gamma_\alpha$ (resp. $\omega_\alpha$) induces an element in $H_{2n-1}^{\prim}(X, \Z)$ (resp. $H^{2n-1}_{\prim}(X, \Z)$) which is the lattice of $J_{\prim}(X)$. Hence, the topological Abel--Jacobi mapping $\opname{a}_{X,t}^{\opname{top}}$ is well-defined.

  In order to obtain the surjectivity property for the topological Abel--Jacobi mapping $ \opname{a}_{X,t}^{\opname{top}}$, we need to use all hyperplane sections of $X$. Therefore, consider the local system $\opname{R}^{2n-2}_{\van}\p_{X*}^{\sm}\Z$ over $\O_X^{\sm}$ which are
fiberwisely given by $H^{2n-2}_{\van}(X_{t}, \Z)$ and denote the total space by $ \big|\opname{R}^{2n-2}_{\van}\p_{X*}^{\sm}\Z\big| $,  
 then the topological Abel--Jacobi mapping becomes
 $$ \opname{a}_X^{\opname{top}}: \big|\opname{R}^{2n-2}_{\van}\p_{X*}^{\sm}\Z\big| \to J_{\prim}(X).$$
  X.\ Zhao proved that $\opname{a}_X^{\opname{top}}$ is real analytic \cite{TopAJ}. However, $ \opname{a}_X^{\opname{top}}$ is not holomorphic, which is the drawback of this mapping. In the future, we will explore a possible direction for overcoming this drawback.

Fix $t_0\in \O_X^{\sm}$ and  $\alpha_0\in H_{2n-2}^{\van}(X_{t_0}, \Z)$, we should consider the induced homomorphism on fundamental groups called the \emph{topological Abel--Jacobi homomorphism} in this paper,
   $$ \opname{a}_{X*}^{\opname{top}}: \pi_1\left(\big|\opname{R}^{2n-2}_{\van}\p_{X*}^{\sm}\Z\big|, (t_0, \alpha_0)\right) \to \pi_1(J_{\prim}(X))\cong  H_{2n-1}^{\prim}(X, \Z).$$

 Denote by $G = \pi_1(\O_X^{\sm}, t_0)$ and $V =H_{2n-2}^{\van}\left(X_{t_0}, \Z\right)$. The group $G$ can act on $V$ by the following monodromy representation:
 \begin{align*}
\rho_{\van} :& \pi_1(\O_X^{\sm}, t_0) \to \Aut\left(H_{2n-2}^{\van}\left(X_{t_0}, \Z\right)\right).
\end{align*}
For any $\alpha\in V$, denote by $G_\alpha$ the stabilizer of $\alpha$ in $G$. For any $g\in G$, denote by $V^g$ the set of all fixed points of $g$ in $V$.
 Note that $\big|\opname{R}^{2n-2}_{\van}\p_{X*}^{\sm}\Z\big|$  is an infinite-sheeted covering space of $\O_X^{\sm}$, then any element in the fundamental group
  $\pi_1\left(\big|\opname{R}^{2n-2}_{\van}\p_{X*}^{\sm}\Z\big|, (t_0, \alpha_0)\right)$ can be realized as an element $g\in G$ such that $g\alpha=\alpha$, \ie,
  $$\pi_1\left(\big|\opname{R}^{2n-2}_{\van}\p_{X*}^{\sm}\Z\big|, (t_0, \alpha_0)\right)= G_{\alpha_0}.$$ 
 Therefore, $ \opname{a}_{X*}^{\opname{top}}$ can be extended by addition to
 \begin{equation}\label{top-map}
  \opname{a}_{X*}^{\opname{top}}: \bigoplus\limits_{\alpha \in V}G_{\alpha} \to  H_{2n-1}^{\prim}(X, \Z).
  \end{equation}  
Let us give a geometric description of (\ref{top-map}). For any $\alpha\in V$ and $g\in G_\alpha$, let $\gamma_\alpha$ be a $(2n-2)$-cycle on $X_{t_0}$ representing $\alpha$,
 then we have $$[g\gamma_\alpha]=g\alpha=\alpha=[\gamma_\alpha],$$  thus there is a $(2n-1)$-chain $A$ in $X_{t_0}$ such that $\partial A=g\gamma_\alpha-\gamma_\alpha$.
By the short exact sequence 
$$ 0 \to H_{2n-1}^{\prim} (X, \Z)\to H_{2n-1}(X, X_{t_0}; \Z)  \xrightarrow{\partial} H_{2n-2}^{\van}(X_{t_0}, \Z)\to 0,$$
there is a relative $(2n-1)$-chain
\begin{equation}\label{rel-class}
\beta=[\gamma_\beta] \in H_{2n-1}(X, X_{t_0}; \Q)
\end{equation}
such that $\partial\gamma_\beta=\gamma_\alpha$.
The group $G$ can also act on $H_{2n-1}\left(X, X_{t_0}; \Z\right)$ by the following monodromy representation:
\begin{align*}
\rho_{\opname{rel}}: &\pi_1(\O_X^{\sm}, t_0) \to \Aut\left(H_{2n-1}\left(X, X_{t_0}; \Z\right)\right).
\end{align*}

For any $g\in G_{\alpha}$, we have  
$$\partial (g\gamma_\beta-\gamma_\beta)=g\partial\gamma_\beta-\partial\gamma_\beta =g\gamma_\alpha - \gamma_\alpha=\partial A,$$
thus $g\gamma_\beta-\gamma_\beta-A$ is a $(2n-1)$-cycle on $X$.  Since the monodromy action on $H_{2n-1}^{\prim} (X, \Z)$ is trivial,   $[g\gamma_\beta-\gamma_\beta]$ is independent of the choice of $\beta$. Taking the ambiguities in choosing $A$ and $\gamma_\beta$ into account, $[g\gamma_\beta-\gamma_\beta-A]$ is only defined up to elements of $H_{2n-1}(X_{t_0},\Z)$, so
we conclude that $\opname{a}_{X*}^{\opname{top}}(\alpha, g)= [g\gamma_\beta-\gamma_\beta-A] \in  H_{2n-1}^{\prim} (X, \Z)$.

Next, recall the tube mapping constructed by C.\ Schnell \cite{Tube}, which is defined as follows:
$$\tau: \bigoplus\limits_{g\in  G}V^g \to  H_{2n-1}^{\prim}(X,\Z). $$
For any $g  \in  G$ and $\alpha\in V^g$,  similarly,  let $\gamma_\alpha$ be a $(2n-2)$-cycle on $X_{t_0}$ representing $\alpha$, then we have $$[g\gamma_\alpha]=g\alpha=\alpha=[\gamma_\alpha],$$  thus there is a $(2n-1)$-chain $A$ in $X_{t_0}$ such that $\partial A=g\gamma_\alpha-\gamma_\alpha$. Transporting $\gamma_\alpha$ along $g$ gives us a $(2n-1)$-chain $B$ in $X$ with boundary
$$\partial B = g\gamma_\alpha-\gamma_\alpha=\partial A,$$
 then $B-A$ is a $(2n-1)$-cycle on $X$. 
Taking the ambiguities in choosing $A$ into account, $[B-A]$ is only defined up to elements of $H_{2n-1}(X_{t_0},\Z)$, so $\tau(g, \alpha)=[B-A] \in H_{2n-1}^{\prim}(X, \Z)$. We call $\tau(g, \alpha)$ the tube class determined by $g$ and $\alpha$.

Note that transporting $\gamma_\beta$ in (\ref{rel-class}) along $g$ gives us a $2n$-chain $\Gamma$ in $X$ with boundary
$$\partial \Gamma = (g\gamma_\beta-\gamma_\beta) - B= (g\gamma_\beta-\gamma_\beta -A) - (B- A),$$
 thus
$\a_{X*}^{\opname{top}}(\alpha, g)= \tau(g, \alpha).$
At the same time,  we have the following identity
$$\bigoplus\limits_{\alpha \in V }G_{\alpha} =\bigoplus\limits_{g\in  G} V^g=\bigoplus\limits_{g\in  \pi_1(\O_X^{\sm}, t_0)}\left\{\alpha \in H_{2n-2}^{\van}\left(X_{t_0}, \Z\right)\, | \, g\alpha=\alpha\right\},$$
so $ \opname{a}_{X*}^{\opname{top}}$ coincides with the tube mapping $\tau$ up to a sign.

\subsection{Conjectures} 
\label{s:elt-conj}

Schnell \cite[Theorem 1]{Tube} proved that $\tau$ is surjective over $\Q$  if $V \neq 0$, so is $ \opname{a}_{X*}^{\opname{top}}$.
Clearly, in order to generate $H_{2n-1}^{\prim}(X, \Q)$, we don't need to use the whole $V$.
  Naturally, we want to know the minimum number of  elementary vanishing cycles in $V$ to generate  $H_{2n-1}^{\prim}(X, \Q)$.
  However, this problem is unsolved even for $n=1$ case.
We conjecture that only one elementary vanishing cycle  
will be enough to generate $H_{2n-1}^{\van}(X, \Q)$, which is a summand of $H_{2n-1}^{\prim}(X, \Q)$ (see (\ref{eqn-summand-2})). 

If $\alpha$ is an elementary vanishing cycle, then $\gamma_\alpha$ is a vanishing sphere and $\gamma_\beta$ is the cone over $\gamma_\alpha$.  For any $g\in G_\alpha$,    $g\gamma_\alpha$ is also a vanishing sphere and $g\gamma_\beta$ is the  cone over $g\gamma_\alpha$.
We can paste $g\gamma_\beta$ with $\gamma_\beta$ along $A$
since $\partial A=g\gamma_\alpha-\gamma_\alpha$, which gives us a $(2n-1)$-vanishing sphere $g\gamma_\beta-\gamma_\beta-A$, so
$$\a_{X*}^{\opname{top}}(\alpha, g)=[g\gamma_\beta-\gamma_\beta-A]\in   H_{2n-1}^{\van}(X, \Z).$$
Using an elementary vanishing cycle, we conjecture a strong tube theorem compared with Schnell's tube theorem.
 \begin{conjecture}[Strong Tube Theorem]\label{Conj-old}
 Let $X$ be a smooth hyperplane section of a $2n$--dimensional complex projective manifold and $X_{t_0}$ be the smooth hyperplane section of $X$ over $t_0\in \O_X^{\sm}$. Let
 $\alpha\in H_{2n-2}^{\van}(X_{t_0}, \Z)$ be an elementary vanishing cycle. Denote by $G_\alpha$  the stabilizer of $\alpha$ in $\pi_1(\O_{X}^{\sm}, t_0)$.
 Then the image of the topological Abel--Jacobi homomorphism
 $$\a_{X*}^{\opname{top}}: G_{\alpha} \to H_{2n-1}^{\van}(X, \Z),$$
 is $m H_{2n-1}^{\van}(X, \Z)$ for some positive integer $m$.
 \end{conjecture}

Since $\a_{X*}^{\opname{top}}$ is a group homomorphism and
$H_{2n-1}^{\van}(X, \Z)$ is an Abelian group,  $\a_{X*}^{\opname{top}}$ factors through the Abelianization ${G_{\alpha}}/{[G_{\alpha}, G_{\alpha}]}$ of $G_{\alpha}$. We still denote by
 $$\a_{X*}^{\opname{top}} : \frac{G_{\alpha}}{[G_{\alpha}, G_{\alpha}]} \to H_{2n-1}^{\van}(X, \Z).$$
Note that $H_1(G_{\alpha}, \Z)\cong  {G_{\alpha}}/{[G_{\alpha}, G_{\alpha}]}$, 
then
Conjecture \ref{Conj-old} can be restated as follows.

 \begin{conjecture}[Strong Tube Theorem]\label{Conj}
Let $X$ be a smooth hyperplane section of a $2n$--dimensional complex projective manifold and $X_{t_0}$ be the smooth hyperplane section of $X$ over $t_0\in \O_X^{\sm}$. Let
 $\alpha\in H_{2n-2}^{\van}(X_{t_0}, \Z)$ be an elementary vanishing cycle. Denote by $G_\alpha$  the stabilizer of $\alpha$ in $\pi_1(\O_{X}^{\sm}, t_0)$.
 Then the image of the topological Abel--Jacobi homomorphism 
   $$\a_{X*}^{\opname{top}}: H_1(G_{\alpha}, \Z) \to H_{2n-1}^{\van}(X, \Z)$$
  is $mH_{2n-1}^{\van}(S_t, \Z)$ for some positive integer $m$. Therefore,
   $$\a_{X*}^{\opname{top}}: H_1(G_{\alpha}, \Q) \to H_{2n-1}^{\van}(X, \Q)$$
  is surjective.
 \end{conjecture}

Furthermore, denote by $(\G_X^{\sm}, u)$ the connected covering space of $(\O^{\sm}_X, t_0)$ corresponding the stabilizer $G_{\alpha}=\pi_1\left(\big|\opname{R}^{2n-2}_{\van}\p_{X*}^{\sm}\Z\big|, (t_0, \alpha)\right)$.
By the classification theorem of covering spaces, $\G_X^{\sm}$ is exactly the connected component of $\big|\opname{R}^{2n-2}_{\van}\p_{X*}^{\sm}\Z\big|$ containing $(t_0, \alpha)$.
Denote by $\G_X$ the desingularization of the compactification of $\G_X^{\sm}$.
By composing the topological Abel--Jacobi map with the projection $J_{\prim}(X) \to J_{\van}(X)$, we get
$$\opname{a}_X^{\opname{top}}: \big|\opname{R}^{2n-2}_{\van}\p_{X*}^{\sm}\Z\big|  \to J_{\van}(X).$$
Naturally, we conjecture the Jacobi--type Theorem that several copies of $\G_X$ will fill in the vanishing intermediate Jacobian $J_{\van}(X)$.
\begin{conjecture}[Jacobi--type Theorem]\label{Conj-AJ}
Let $X$ be a smooth hyperplane section of a $2n$--dimensional complex projective manifold.
There is some positive integer $k$ such that the following topological Abel--Jacobi mapping
$$\a_X^{\opname{top}} : \G_X^{(k)}:=\G_X \times_{\O_X} \G_X\times_{\O_X} \cdots \times_{\O_X}\G_X \to J_{\van}(X)$$
is surjective for all $k\geq K$.
\end{conjecture}

\section{Strong tube theorem for complex project curves}
\label{sec-tube}
Let $S$ be a nondegenerate irreducible complex projective surface of degree $d$ in a projective space $\P^N$ and $X$ be a smooth hyperplane section of $S$.
Pick a general pencil $\L$ of hyperplane sections of $X$, which forms a Lefschetz pencil \cite[Corollary 2.10]{Voisin2}. Denote by $\L^{\sm}$ the subset of $\L$ parametrizing all smooth hyperplane sections. Fix a point $t_0\in \L^{\sm}$, denote by $X_{t_0}$ the fiber over $t_0$,  then we have a monodromy representation:
$$\rho_{X}: \pi_1(\L^{\sm}, t_0) \to \Aut(H_0^{\van}(X_{t_0}, \Z)).$$
Pick an elementary vanishing cycle $\alpha\in H_0^{\van}(X_{t_0}, \Z)$  corresponding to $t_1\in \L^{\sg}:=\L\setminus \L^{\sm}$, \ie, there are two distinct points $x_1^+, x_1^- \in X_{t_0}$ such that  $\alpha=[x_1^+-x_1^-]$ and $x_1^+, x_1^-$ come together when $t_0$ moves to $t_1$ under the monodromy action $\rho_X$.
Let $G_{\alpha}$ be the stabilizer of $\alpha$ in $\pi_1(\L^{\sm},  t_0)$ under the action of $\rho_{X}$.
 Then we have the topological Abel--Jacobi homomorphism
 $$\a_{X*}^{\opname{top}}: H_1(G_{\alpha},\Z) \to H_{1}^{\van}(X, \Z).$$
 In this section, we will prove the Strong Tube Theorem \ref{thm-tube} for $X$, \ie,  $\Im \a_{X*}^{\opname{top}}= mH_{1}^{\van}(X, \Z)$ for some nonzero integer $m$ and then $\a_{X*}^{\opname{top}}$ is surjective over $\Q$,
equivalently,  all 1-cycles in $X$ traced out by
 $\gamma\in G_\alpha$ generate  $H_1^{\van}(X, \Q)$.
We will first prove the following conclusion about $\a_{X*}^{\opname{top}} $.
 \begin{lemma}\label{lem-top-1}
 Let $X$ be a nondegenerate complex projective curve in $\P^{N-1}$. Let $\L$ be a Lefscheta pencil of $X$ and $\alpha$ be an elementary vanishing cycle on a smooth hyperplane section over $\L$.
 The image of the topological Abel--Jacobi homomorphism
 $$\a_{X*}^{\opname{top}} : H_1(G_{\alpha}, \Z) \to H_{1}(X, \Z)$$
 is nonzero and independent of the choice of the Lefschetz pencil $\L$ and the elementary vanishing cycle $\alpha$.
 \end{lemma}

Next, since $X$ is a smooth hyperplane section of $S$, we can
pick a general line $\L_S \subseteq (\P^N)^\vee:=\O_S$ passing through $X$, which forms a Lefschetz pencil \cite[Corollary 2.10]{Voisin2}.
Consider the incidence variety
$$S_{\L_S}=\left\{(x, t)\in S\times \L_S \,|\, t(x)=0\right\}.$$
Denote by $\p_S : S_{\L_S} \to \L_S$ the second projection and $S_t$ the hyperplane section over $t\in \O_S$.
Since $\L_S$ passes through $X$ and $X$ is smooth, then there is some point
 $r_0\in \L_S^{\sm}$ such that $X=S_{r_0}$.
 We have the monodromy representation endowed by the Gauss-Manin connection
$$\rho_S : \pi_1(\L^{\sm}_S,  r_0) \to  \Aut\left(H_{1}^{\van}(X, \Z)\right),$$
and the following conclusion about the stable $\Z$-submodules of $H_{1}^{\van}(X,\Z)$ under this monodromy action.
\begin{lemma}\label{lem-top-2}
The stable  $\Z$-submodules of $H_{1}^{\van}(X,\Z)$ under the monodromy action $\rho_S$
are $$mH_{1}^{\van}(X,\Z), \qquad m\in \Z.$$
\end{lemma}
More generally, we will prove this lemma for any middle vanishing (co)homology of an odd dimensional variety  in  Section \ref{s:stable}.

By Lemma \ref{lem-top-1}, we know that $\Im \a_{X*}^{\opname{top}}$ is a nonzero stable  $\Z$-submodules of $H_{1}^{\van}(X,\Z)$ under the monodromy action  $\rho_S$,
 thus  $\Im \a_{X*}^{\opname{top}}=mH_{1}^{\van}(X, \Z)$ for some nonzero integer $m$ by Lemma \ref{lem-top-2}.
Therefore, we only need to prove Lemma \ref{lem-top-1} and Lemma \ref{lem-top-2} in order to obtain the Strong Tube Theorem \ref{thm-tube}.

 \subsection{Proof of Lemma \ref{lem-top-1}}\label{s:restate}
 Let $X$ be a nondegenerate irreducible complex projective curve of genus $g$ and  degree $d$ in  $\P^{N-1}$.
 Note that the base locus $B$ of a Lefschetz pencil is of codimension 2 in $X$, thus $B=\emptyset$ and the incidence variety
$$X_\L=\left\{(x, t)\in X\times \L \, |\, t(x)=0 \right\}\cong
 \opname{Bl}_BX =X.$$
Denote by $\p: X\cong X_\L \to \L$ the second projection which is a branched covering map of degree $d$ with branch locus $\L^{\sg}$ and $X_t:=\p^{-1}(t)$ the hyperplane section  over $t\in \L$.    Since every hyperplane section $X_t$ has at most one ordinary double point as singularity,  the ramification index of each branched point is at most 1.   By Riemann--Hurwitz formula, we have
$$2g-2=d(-2)+\deg\L^{\sg},$$
then $e:=\deg \L^{\sg}=2(d+g-1)$ is an even number.  Denote by $\L^{\sg}=\{t_1, t_2, \cdots, t_e\}$.
We can make a double cover $\sigma: \tilde\L\to \L$ branched along $\L^{\sg}$. 
Let $\tilde X$ be the normalization of the fiber product $X \times_{\L} \widetilde\L$. We have the following commutative diagram
 \begin{equation*}
 \xymatrix{\tilde X\times_{\tilde\L}\tilde X\setminus \Delta\ar[r]^-{\zeta}\ar[d]_{\widehat\p}\ar@{.>}[dr]^{\opname{q}}&\tilde X\ar[r]^{\tilde\sigma}\ar[d]^{\tilde\p}
& X\ar[d]^{\p}\\
 \tilde X\ar[r]^{\tilde\p}&\widetilde\L\ar[r]^{\sigma}& \L},
 \end{equation*}
where $\Delta$ denotes the diagonal of the fiber product $\tilde X\times_{\tilde\L}\tilde X$.
Note that
$\tilde X$ is a smooth irreducible connected curve and $\tilde \p$ is an unbranched covering map of degree $d$, then ${\tilde{X}}_{t_1}={\tilde{\p}}^{-1}(t_1)$ is diffeomorphic to $X_{t_0}$. Note that there is no monodromy action on ${\tilde{X}}_{t_1}$ around each point in $\L^{\sg}$,  then $\rho_X$ is equivalent to the monodromy representation $\tilde\rho_X$ induced by $\tilde\p$ :
$$\tilde\rho_X: \pi_1(\tilde\L, t_1) \twoheadrightarrow \Aut\left(H_{0}^{\van}(\tilde X_{t_1}, \Z)\right),$$
which is surjective because $\rho_X$ is surjective due to the singularity type of $\L^{\sg}$.
Denote by
\begin{itemize}
\item $X_{t_1}=\{ 2x_1, y_{1}, \cdots, y_{d-2}\}$, where $x_1$ the unique singular point on $X_{t_1}$;
\item $\tilde X_{t_1}=\{x_1^+, x_1^-, y_{1}, \cdots, y_{d-2}\}$, where $\{x_1^+, x_1^-\}=\tilde\sigma^{-1}(x_{1})$;
\end{itemize}
then $\alpha=[x_1^+ - x_1^-]\in H_{0}^{\van}(\tilde X_{t_1}, \Z)$. Let $\tilde G_{(x_1^+, x_1^-)}$ be the stabilizer of $[x_1^+ - x_1^-]$ in $\pi_1(\tilde\L, t_1)$, then it plays the same role as $G_{\alpha}$.

Denote by $\Xi_{t_1}:=\opname{q}^{-1}(t_1)$ the set of all ordered pairs of distinct points in $\tilde X_{t_1}$. 
Note that $H_{0}^{\van}(\tilde X_{t_1}, \Z)$ is generated by the difference of two distinct points in $\tilde X_{t_1}$, then $\tilde\rho_X$ is the same as the following monodromy representation which is still denoted by $\tilde\rho_X$:
\begin{equation}\label{isomorphism}
\begin{aligned}
\tilde\rho_X: \pi_1(\tilde{\L}, t_1) \to \opname{Aut}\left(\Xi_{t_1}\right)\cong S_{d(d-1)},
\end{aligned}
\end{equation} where  $S_{d(d-1)}$ denotes the symmetric group on $d(d-1)$ letters.
The isomorphism in (\ref{isomorphism}) is due to $|\Xi_{t_1}|=d(d-1)$ because $\opname{q}$ is  an unbranched covering map of degree $d^2-d$.
Let $\opname{Aut}\left(\Xi_{t_1}\right)_{(x_1^+,x_1^-)}$ be the stabilizer of the ordered pair $(x_1^+, x_1^-)$ in
 $\opname{Aut}\left(\Xi_{t_1}\right)$, 
 then $$\opname{Aut}\left(\Xi_{t_1}\right)_{(x_1^+,x_1^-)}\cong S_{d(d-1)-1} \qquad \text{and}\qquad
\tilde G_{(x_1^+,x_1^-)}={\tilde\rho_X}^{-1}\left(\opname{Aut}\left(\Xi_{t_1}\right)_{(x_1^+,x_1^-)}\right).$$ 
 Note that $\rho_X$ induces an injective group homomorphism
 $$\frac{\pi_1(\tilde{\L}, t_1)}{\tilde G_{(x_1^+,x_1^-)}} \hookrightarrow \frac{\opname{Aut}\left(\Xi_{t_1}\right)}{\opname{Aut}\left(\Xi_{t_1}\right)_{(x_1^+,x_1^-)}}\cong \frac{S_{d(d-1)}}{S_{d(d-1)-1}},$$
 so the group index
 $$\left[\pi_1(\tilde{\L}, t_1):  \tilde G_{(x_1^+,x_1^-)}\right]\leq \frac{|S_{d(d-1)}|}{|S_{d(d-1)-1}|}=\frac{[d(d-1)]!}{[d(d-1)-1]!}=d(d-1).$$
Let $\kappa : (\F, u_1) \to (\tilde{\L}, t_1)$ be the path-connected unbranched covering space of $\tilde{\L}$ corresponding to $\tilde G_{(x_1^+,x_1^-)}$, \ie, $\pi_1 (\F, u_1) \cong  \tilde G_{(x_1^+,x_1^-)}$, then
$$\deg \kappa = \left[\pi_1(\tilde{\L}, t_1): \tilde G_{(x_1^+,x_1^-)}\right] \leq d(d-1).$$
Since 
$\tilde G_{(x_1^+,x_1^-)}$ 
 is a subgroup of $\tilde\p_{*}\left(\pi_1\left(\tilde{X}\times_{\tilde{\L}}\tilde{X}\setminus \Delta, (x_1^+, x_1^-)\right)\right)$, by the lifting criterion, 
 $\kappa$ factors through $\opname{q}$, \ie, we have the following commutative diagram
 \begin{equation*}
\xymatrix{&\left(\tilde{X}\times_{\tilde{\L}}\tilde{X}\setminus\Delta, (x_1^+, x_1^-)\right)\ar[d]_{\opname{q}} \ar[r]^>>>>>{\xi} &X^+\times X^- \ar[r]^{\a_X}& J(X)\\
(\F, u_1)\ar[r]^{\kappa}\ar@{.>}[ru]^{\eta}\ar@{.>}[rrru]|-{\theta}&(\tilde{\L}, t_1)},
\end{equation*}
where $\a_X$ denotes the canonical Abel--Jacobi mapping
\begin{equation*}
\begin{aligned}
\opname{a}_X: X^+\times X^- &\to J(X)\\
(x^+, x^-)&\mapsto \int^{x^+}_{x^-} .
\end{aligned}
\end{equation*}
In fact,  $\tilde G_{(x_1^+,x_1^-)}=\tilde\p_{*}\left(\pi_1\left(\tilde{X}\times_{\tilde{\L}}\tilde{X}\setminus \Delta, (x_1^+, x_1^-)\right)\right)$, thus $\F$ is a connected component of $\tilde{X}\times_{\tilde{\L}}\tilde{X}\setminus\Delta$ including $(x_1^+, x_1^-)$.
Denote by $\theta=\opname{a}_X\circ \xi\circ\eta: \F \to J(X)$
which induces a homomorphism on homology
 \begin{equation*}
 \theta_{*}: H_1(\F,\Z) \to H_1(J(X),\Z) \xrightarrow{\sim} H_1(X, \Z) . 
 \end{equation*}
Since $\pi_1(\F, u_1)\cong \tilde G_{(x_1^+,x_1^-)}$ and $H_1(\F, \Z)=\pi_1(\F, u_1)/[\pi_1(\F, u_1),\pi_1(\F, u_1)]$, we have
$$H_1(\tilde G_{(x_1^+,x_1^-)}, \Z) =\tilde G_{(x_1^+,x_1^-)}/[\tilde G_{(x_1^+,x_1^-)},\tilde G_{(x_1^+,x_1^-)}] \cong H_1(\F, \Z).$$
Therefore, $\theta_{*}$ is exactly the topological Abel--Jacobi homomorphism $\a^{\opname{top}}_{X*}$, 
then Lemma \ref{lem-top-1} is equivalent to the following two lemmas.
  \begin{lemma}\label{thm-top-2}
 The image of  $ \theta_{*}$ is independent of the choice of the Lefschetz pencil and the elementary vanishing cycle.
 \end{lemma}
  \begin{lemma}\label{thm-top-3}
 The image of  $ \theta_{*}$ is nonzero for any Lefschetz pencil and any elementary vanishing cycle.
 \end{lemma}

 By Lemma \ref{thm-top-2}, we only need to prove Lemma \ref{thm-top-3} for a specific Lefschetz pencil $\L$ and a specific elementary vanishing cycle $\alpha$.  


\medskip

\noindent{\bf Proof of Lemma \ref{thm-top-2}.} 
\label{ss: IndLef}
Denote by $\O_X:=(\P^{N-1})^\vee$ parametrizes all hyperplane sections of $X$
and $\opname{G}:=\opname{Gr}\left(1, \O_{X}\right)$
the Grassmannian, then any point in $\Gr$ represents a pencil of hyperplane sections of $X$. In particular, a Lefschetz pencil $\L$ is a point in $\Gr$. Denote by $\opname{U}$ the set of all Lefschetz pencils, which is an open subset of $\Gr$ by the following proposition.
\begin{proposition}{\cite[Proposition 2.9]{Voisin2}}
Let $X$ be a smooth nondegenerate subvariety of $\P^N$. Then a pencil of hyperplane sections $\{X_t\}_{t\in \P^1}$ is a Lefschetz pencil if and only if one of the following two conditions is satisfied.
\begin{enumerate}
\item $X^\vee$ is a hypersurface in $(\P^N)^\vee$ and the corresponding line $\P^1$ meets $X^\vee$ transversally in the open dense set $X^\vee_o$ which denotes the subset of $X^\vee$ parametrizing the hyperplanes $H$ such that $X\cap H$ has at most one ordinary double point as singularity.
\item $\dim X^\vee\leq N-2$ and the corresponding line $\P^1$ does not meet $X^\vee$ in $(\P^N)^\vee$.
\end{enumerate}
where $X^\vee:=\{ H\in (\P^N)^\vee\, \vert \, H\cap X \text{ is singular}\}$ is the dual variety of $X$.
\end{proposition}
Note that we are in the first case because there is no vanishing homology in the second case, so $\opname{U}\cong X^\vee_o$ is an open subset of $\Gr$.

Pick $\mathbb{L}\in \opname{U}$ and write $X^\vee\cap \L=\{t_1, \cdots, t_e\}$, where $e$ is a positive even integer.  Let $\alpha_k$ be an elementary vanishing cycle corresponding to $t_k$.  By the analysis in Section \ref{s:restate}, we have the following commutative diagram:
 \begin{equation*}
\xymatrix{&\left(\tilde{X}\times_{\tilde{\L}}\tilde{X}\setminus\Delta, (x_k^+, x_k^-)\right)\ar[d]_{\opname{q}_{\L}} \ar[r]^>>>>>{\xi_{\L}} &X^+\times X^- \ar[r]^{\a_X}& J(X)\\
(\F_{\L}, u_k)\ar[r]^{\kappa_{\L}}\ar@{.>}[ru]^{\eta_{\L}}\ar@{.>}[rrru]|-{\theta_{\L}}&(\tilde{\L}, t_k)\ar[r]^{\sigma_{\L}}& (\L, t_k)}.
\end{equation*}
Note that $H_1(X, \Z): =\bigoplus_{i=1}^{2g}\Z\gamma_i\cong \Z^{2g}$ is a finitely generated free Abelian group, then any subgroup $A$ is free generated by
$$d_1\gamma_1, \,\quad d_1d_2\gamma_2, \,\quad  \cdots, \, \quad d_1d_2\cdots d_{2g}\gamma_{2g},$$
 where 
$d_1, d_2, \cdots, d_{2g}$ are nonnegative integers. The sequence $(d_1, d_2, \cdots, d_{2g})$  depends only on the group $H_1(X, \Z)$ and the subgroup $A$, and not on the particular basis $\{\gamma_1, \cdots, \gamma_{2g}\}$.
Denote by $\N$ the set of nonnegative integers, then the discrete set $\N^{2g}$ parametrizes all subgroups in $H_1(X, \Z)$.
Since $\theta_{\L}$ induces a homomorphism on homology
 \begin{equation*}
 \theta_{\L*}: H_1(\F_{\L},\Z) \to H_1(J(X),\Z) \xrightarrow{\sim} H_1(X, \Z),
 \end{equation*}
thus $\im \theta_{\L*}$ is a subgroup of $H_1(X, \Z)$.
Therefore, we can define
\begin{equation*}
\begin{aligned}
\varepsilon: \opname{U} &\to  \N^{2g}\\
\L &\mapsto\varepsilon(\mathbb{L})=\im \theta_{\L*}.
\end{aligned}
\end{equation*}
Clearly, $\im \theta_{\L*}$ does not depend on the choice of the Lefschetz pencil $\L$ if and only if $\varepsilon$ is a constant map.
Since $\opname{U}$ is connected and the quotient topology of $\N^{2g}$ is discrete,  $\varepsilon$ is constant if and only if it is continuous.
Hence, we will prove that $\varepsilon$ is continuous in the rest of this section.

Consider the incidence variety
$$ \xymatrix{\mathfrak{I}=\left\{(\L, t)\in \Gr \times \O_{X} \, |\, t\in \L\right\} \ar[r]^>>>>{\pr_2}\ar[d]^{\pr_1}& \O_X\\ \Gr}.$$
Denote by $\mathfrak{I}_{\opname{U}}:=\opname{pr}_1^{-1}(\opname{U})=\left\{(\L, t)\in \opname{U}\times \O_{X} \, |\, t\in \L\right\}$.
If the following conditions
\begin{itemize}
\item \textbf{Cond 1.} \, $\pr_1: \I_{\opname{U}} \to \opname{U}$ is a proper submersion;
\item  \textbf{Cond 2.}  \, $\pr_2^{-1}(X^\vee) \cap \I_{\opname{U}}$ is a closed submanifold;
\item  \textbf{Cond 3.}  \, $\pr_1 : \pr_2^{-1}(X^\vee) \cap \I_{\opname{U}} \to \opname{U}$ is also a submersion;
\end{itemize}
hold,  we can apply Ehresmann's Fibration Theorem \cite{DiffTop, Ehresmann, HodgeTheory} to $\pr_1 : \I_{\opname{U}} \to \opname{U}$.
\begin{theorem}[Ehresmann's Fibration Theorem]
 Let $f: X\to B$ be a proper submersion between the $C^k$-manifolds $X$ and $B$. Then $f$ is a locally trivially fibration. \ie , for every point $b \in B$, there is a neighborhood $U_b$ of $b$ and a $C^k$-diffeomorphism $\psi_b: U_b\times f^{-1}(b) \xrightarrow{\sim} f^{-1}(U_b)$ such that $f\circ \psi_b=\pr_1$. Moreover, if  $Y \subseteq  X$ is a closed submanifold such that $f|_Y$ is still a submersion, then $f$ fibers $X$ locally trivial over $Y$ , \ie , the diffeomorphism $\psi_b$ above can be chosen to carry $U_b\times \left(f^{-1}(b)\cap Y\right)$ onto $f^{-1}(b)\cap Y$.
\end{theorem}
We obtain that $\pr_1 : \I_{\opname{U}} \to \opname{U}$ is a locally trivially fibration and it fibers  $\pr_2^{-1}(X^\vee) \cap \I_{\opname{U}}$  locally trivial over $\opname{U}$, \ie, for any $\L\in \opname{U}$, there is a neighborhood $V_{\L}$ of $\L$ and a $C^\infty$-diffeomorphism $\psi_{\L} : V_{\L}\times \pr_1^{-1}(\L) \xrightarrow{\sim}\pr_1^{-1}(V_{\L})$ such that $\pr_1\circ \psi_{\L}=\pr_1$, moreover, $\psi_{\L}$ carries $V_{\L}\times \left(\pr_1^{-1}(\L)\cap \pr_2^{-1}(X^\vee)\cap \I_{\opname{U}}\right)$ onto $\pr_1^{-1}(V_{\L})\cap \pr_2^{-1}(X^\vee)\cap \I_{\opname{U}}$.
That is to say, we have a trivialization $(V_\mathbb{L},\psi_{\mathbb{L}})_{\mathbb{L}\in\opname{U}}$ of $\I_{\opname{U}} \xrightarrow{\opname{pr}_1} \opname{U}$ such that the restriction of the trivialization to $X^\vee$ also gives a trivialization $(V_{\mathbb{L}},\psi_{\mathbb{L}}^r)_{\mathbb{L}\in \opname{U}}$ of $\I_{\opname{U}, X^\vee} \xrightarrow{\opname{pr}_1}\opname{U}$.

On $V_{\mathbb{L}}$,  we have $\psi_{\mathbb{L}} : \I_{\opname{U}}|_{V_{\mathbb{L}}} \cong V_{\mathbb{L}} \times \P^1$ and $\psi_{\mathbb{L}}^r : \I_{\opname{U}, X^\vee}|_{V_{\mathbb{L}}} \cong V_{\mathbb{L}} \times \{t_1,\cdots, t_e\}$. 
When $(\L',  t_1',\cdots, t_e')$ continuously goes to $(\L, t_1,\cdots, t_e)$ on $V_{\L}$, the double cover $\tilde\L'$ branched at $\L'\cap X^\vee$ continuously goes to $\tilde \L$, which force $\F_{\L'}$ to move continuously to $\F_{\L}$ and the elementary vanishing cycle $\alpha_k'$ corresponding to $t_k'$ to move continuously to $\alpha_k$ corresponding to $t_k$. Hence,
$\varepsilon(\mathbb{L}')=\im \theta_{\mathbb{L}'*}$ continuously goes to $\varepsilon(\mathbb{L})=\im \theta_{\mathbb{L}*} $, which means $\varepsilon$ is continuous at $\L$. Since $\L$ is arbitrary, $\varepsilon$ is continuous.

Therefore, in order to prove $\varepsilon$ is continuous, we only need to show the above three conditions: \textbf{Cond 1}, \textbf{Cond 2} and \textbf{Cond 3}.
\medskip

\noindent\textit{Proof of \textbf{Cond 1.} }
As we know \cite{GH-Principles},  $\Gr=\left(1, \O_{X}\right)$ is a compact complex manifold and $\Gr=\opname{M}_{2\times (N+1)}(\C)/\GL_2(\C)$. Any $\L\in \Gr$ can be uniquely represented by a $2\times (N+1)$ matrix whose some $2 \times 2$ minor is the identity matrix.
For any $1\leq i_1<i_2\leq N+1$, denote by
\begin{equation*}
\begin{aligned}
U_{\{i_1,i_2\}}
=\left\{ \L\in \Gr \,\Big{|} \,
\begin{array}{@{}c@{}c@{}c@{}c@{}c@{}c@{}c@{}c@{}c@{}c@{}c}
     && & i_1&&&&i_2&&& \\
  \left. \begin{array}{c} \\  \end{array}
   \L= \right(
                     \begin{array}{c} * \\ * \end{array}
                    & \begin{array}{c} \cdots \\ \cdots \end{array}
                    & \begin{array}{c} * \\ * \end{array}
                    & \begin{array}{c} 1 \\ 0 \end{array}
                    & \begin{array}{c} * \\ * \end{array}
                         &\begin{array}{c} \cdots \\ \cdots \end{array}
                         & \begin{array}{c} * \\ * \end{array}
                       & \begin{array}{c} 0 \\ 1 \end{array}
                       & \begin{array}{c} * \\ * \end{array}
                       & \begin{array}{c} \cdots \\ \cdots \end{array}
                       & \begin{array}{c} * \\ * \end{array}
                   \left)\begin{array}{c} \\  \end{array}
                           \right.
                           \end{array}\right\} \cong \C^{2N-2},
\end{aligned}
\end{equation*}
then $\{U_{\{i_1,i_2\}}\}_{1\leq i_1<i_2\leq N+1}$ is an open cover of $\Gr$. Define
\noindent\begin{equation*}
\begin{aligned}
\varrho_{\{1,2\}} : U_{\{1,2\}} \times \P^1 & \to \pr_1^{-1}(U_{\{1,2\}})\\
\left(\small{\begin{pmatrix}1&0&a_2&\cdots&a_N\\ 0&1&b_2&\cdots & b_N\end{pmatrix}}, [a, b]\right) &\mapsto
\left(\small{\begin{pmatrix}1&0&a_2&\cdots&a_N\\ 0&1&b_2&\cdots & b_N\end{pmatrix}},
[a, b, aa_2+bb_2, \cdots, aa_N+bb_N]\right),
\end{aligned}
\end{equation*}
where $\pr_1: \I \to \Gr$. Clearly, $\varrho_{\{1,2\}}$ is biholomorphic. Similarly, we can define $\varrho_{\{i_1,i_2\}}$ for all $1\leq i_1<i_2\leq N+1$. Note that $\{U_{\{i_1,i_2\}}, \varrho_{\{i_1,i_2\}}\}_{\{1\leq i_1<i_2\leq N+1\}}$ is a local trivialization of $\pr_1$, thus $\I$ is a holomorphic $\P^1$-bundle over $\Gr$. Since  $\Gr$ is a compact complex manifold,  $\I$ is also a compact complex manifold. Note that $\pr_1$ is projective, thus it is proper.

Since $\opname{U}$ is an open subset in $\Gr$, then $\I_{\opname{U}}=\pr_1^{-1}(\opname{U})$ is a holomorphic $\P^1$-bunlde over $\opname{U}$. Therefore,  $\pr_1|_{\opname{U}} : \I_{\opname{U}} \to \opname{U}$ is a submersion and $\pr_1|_{\opname{U}} $ is also proper because $\pr_1$ is proper.

\medskip

\noindent\textit{Proof of \textbf{Cond 2.}}
Let $[X_0, \cdots, X_N]$ be the homogenous coordinate of $\O_{X}$ and
$$U^i=\left\{[X_0, \cdots, X_N]\in \O_{X}\, |\, X_i\neq 0\right\}, \quad i=0, \cdots N, $$
be the canonical open subsets of $\O_{X}$. Consider the following biholomorphism
\begin{equation*}
\begin{aligned}
&\varpi_i : \P^{N-1}\times U^i  \to \pr_2^{-1}(U^i)\\
&\left([Y_1, \cdots, Y_N], [X_0, \cdots, X_{i-1} ,1, X_{i+1} \cdots, X_N]\right)
\mapsto  \\
& \left(\begin{pmatrix}X_0&\cdots&X_{i-1}&1& X_{i+1} &\cdots & X_N\\
Y_1&\cdots&Y_i&0& Y_{i+1} &\cdots & Y_N\end{pmatrix},
[X_0, \cdots, X_{i-1} ,1, X_{i+1} \cdots, X_N]\right).
\end{aligned}
\end{equation*}
 It is clear that
$\left\{U^i, \varpi_i\right\}_{i=0}^N$ is a trivialization of $\pr_2$, then $\I$ is a holomorphic $\P^{N-1}$-bundle over $\O_{X}$. Since $\pr_2$ is projective, it is proper.

Since $X^\vee \subseteq  \O_{X}$ is a hypersurface, $\pr_2^{-1}(X^\vee)$ is a closed submanifold of $\I$ and then it is compact because $\I$ is compact. Therefore,  $\pr_2^{-1}(X^\vee) \cap \I_{\opname{U}}=\I_{\opname{U}}\times_{\O_{X}} X^\vee$ is a closed submanifold of $X^\vee$.

\medskip

\noindent\textit{Proof of \textbf{Cond 3.}}
Let $X^\vee=\left\{[X_0, \cdots, X_N]\in \O_{X}\, |\, F(X_0, \cdots, X_N)=0\right\}$, where $F$ is an irreducible homogenous polynomial of degree $e$ in $\C[X_0, \cdots, X_N]$.
 We have
\begin{equation*}
\begin{aligned}
F(a, b, aa_2+bb_2, \cdots, aa_N+bb_N)=(af_0-bg_0)\cdot\cdots\cdot(af_e-bg_e)=0
\end{aligned}
\end{equation*}
where ${\left({\begin{pmatrix}1&0&a_2&\cdots&a_N\\ 0&1&b_2&\cdots & b_N\end{pmatrix}},
[a, b, aa_2+bb_2, \cdots, aa_N+bb_N]\right)}\in \pr_1^{-1}(U_{\{1,2\}})\cap \pr_2^{-1}(X^\vee)$, $f_i:=f_i(a_2, \cdots, a_N)$ and $g_i:=g_i(b_2, \cdots, b_N) $, then there are exactly $e$ distinct points $P_i := [g_i, f_i] $, $i=1, \cdots, e$, on the fiber $\P^1$ over each point in $U_{\{1,2\}}$. Denote by
\begin{equation*}
\begin{aligned}
&V^i=\left\{\tiny{\left(\small{\begin{pmatrix}1&0&a_2&\cdots&a_N\\ 0&1&b_2&\cdots & b_N\end{pmatrix}},
[g_i, f_i, a_2g_i+b_2f_i, \cdots, a_Ng_i+b_Nf_i]\right)}\right.\\
&\left.  \in \pr_1^{-1}(U_{\{1,2\}})\cap \pr_2^{-1}(X^\vee)\, \big |\,\begin{pmatrix}1&0&a_2&\cdots&a_N\\ 0&1&b_2&\cdots & b_N\end{pmatrix}\in U_{\{1,2\}}\right\}\cong U_{\{1,2\}}, \quad i =1, \cdots, e,
\end{aligned}
\end{equation*}
which are distinct open subsets and $ \pr_1^{-1}(U_{\{1,2\}})\cap \pr_2^{-1}(X^\vee)=\bigcup\limits_{i=1}^e V^i$. Consequently,
$$\pr_1: \pr_2^{-1}(X^\vee) \cap \I_{\opname{U}} \to \opname{U}$$
 is a holomorphic unbranched covering map of degree $e$, so it is a submersion.

\medskip

\noindent\textbf{Proof of Lemma \ref{thm-top-3}.}  
We first prove the differential $d_{u_1}\theta$ at the base point $u_1$ is not zero in Section \ref{ss:diff} and then use this property to show $\theta_{*}$ is not zero in Section \ref{ss:nonzero}.

\subsubsection{Differential of the topological Abel--Jacobi map}\label{ss:diff}
Note that we have the following commutative diagram:
 \begin{equation*}
 \xymatrix{&(\tilde X, x_1^+)\ar[r]^-{\tilde\sigma}\ar[d]^{\tilde\p}
& (X, x_1)\ar[r]^{\a_X}\ar[d]^{\p}& J(X)\\
 (\F, u_1) \ar[r]^{\kappa}\ar@{.>}[ru]^{\widehat\p\circ\eta}&(\widetilde\L, t_1)\ar[r]^{\sigma}& (\L, t_1)},
 \end{equation*}
which means the lifting of $\kappa$ to $(\tilde X, x_1^+)$ exists, by the lifting criterion, we have
$$\pi_1(\F, u_1) \subseteq \pi_1(\tilde X, x_1^+).$$
Since $\tilde X$ is path-connected, $\pi_1(\tilde X, x_1^+) = \pi_1(\tilde X)$ is independent of the choice of the base point $x_1^+$, thus we can pick any point on the fiber $\tilde\p^{-1}=\left\{x_1^+, x_1^-, y_{1}, \cdots, y_{d-2}\right\}$ as the base point to lift $\kappa$, which gives us the following maps:
 \begin{equation*}
 \xymatrix{&(\tilde X,  x_1^{\pm})\ar[r]^-{\tilde\sigma}\ar[d]
& X\ar[r]^{\a_X}& J(X),&&(\tilde X,  y_i)\ar[r]^-{\tilde\sigma}\ar[d]
& X\ar[r]^{\a_X}& J(X)\\
 (\F, u_1) \ar[r]\ar@{-->}[ru]^{f^{\pm}}\ar@{.>}[rru]|>>>>>>>{\phi^{\pm}}\ar@{.>}[rrru]|>>>>>>>>>>>{\varphi^{\pm}}&(\widetilde\L, t_1)&&&
  (\F, u_1) \ar[r]\ar@{-->}[ru]^{f^{i}}\ar@{.>}[rru]|>>>>>>>{\phi^{i}}\ar@{.>}[rrru]|>>>>>>>>>>>{\varphi^{i}}&(\widetilde\L, t_1)},
 \end{equation*}
where $i=1, \cdots, d-2$.  Clearly, $\theta=\varphi^+ - \varphi^-$. Note that  $\deg \tilde\sigma=2$ and $\phi^\lambda$, $\lambda=+, -, 1, \cdots, d-2$ are branched covering maps with branched loci inside of the branched locus $\{y_{1}, \cdots, y_{d-2}\}$ of $\tilde\sigma$. 
We obtain that, for any $i\in \{1, \cdots, d-1\}$,  $u_1$ is a ramification point of $\phi^i$ with ramification index 2, thus there is a small open neighborhood $U_{i}$ of $u_1$ in $\F$ with coordinate $z$ ($u_1$ corresponds to $z=0$) such that  $\phi^i$ is locally given by $\phi^i(z)=z^2$.  Therefore,
$$d_{u_1}\phi^i=\frac{\partial\phi^i}{\partial z}|_{z=0}=2z|_{z=0}=0,$$
 and then $d_{u_1}\varphi^i=0$ for all $i=1, \cdots, d-2$.

However, $d_{u_1} \phi^+\neq 0$ and $d_{u_1} \phi^-\neq 0$. Indeed, since $\phi^+(u_1)=x_1^+$ is an unbranched point of $\phi^+$, there is a small open neighborhood $U_+$ of $u_1$ in $\F$ with coordinate $z$ such that $\phi^+(z)=z$ over $U_+$, so $d_{u_1}\phi^+=1\neq 0$. Similarly, $d_{u_1} \phi^-\neq 0$.

 Moreover, since $\a_X : X \to J(X)$ is an immersion, $d_{x_1}\a_X$ is injective. We obtain that
 $$d_{u_1}\varphi^{\pm}=d_{u_1} ( \a_X\circ\phi^{\pm})=(d_{x_1} \a_X)\circ( d_{u_1}\phi^{\pm} )\neq 0.$$

Let $\varphi=\varphi^++\varphi^-+\varphi^1+\cdots + \varphi^{d-2}$. 
For any $w_1\in \F$, we have
 $$\varphi(w_1)=\a_X\left(\phi^+(w_1)+\phi^-(w_1)+\phi^1(w_1)+\cdots + \phi^{d-2}(w_1)\right).$$
Note that $f^\lambda(w_1)$, $\lambda=+, -, 1, \cdots, d-2$, are points on the same hyperplane section $\tilde\p^{-1}(\kappa({w_1}))$ of  $\tilde X$ and they are distinct by the lifting criterion, so $\left\{f^\lambda(w_1)\right\}_{\lambda = +, -, 1, \cdots, d-2}$
is a permutation of all points in  $\tilde\p^{-1}(\kappa({w_1}))$ and then $\left\{\phi^\lambda(w_1)\right\}_{\lambda = +, -, 1, \cdots, d-2}$ is the hyperplane section $\p^{-1}\left(\sigma\kappa(w_1))\right)$
of $X$ over $\sigma\kappa(w_1)$. Since all hyperplane sections of $X$ come from the same very ample line bundle $\mathcal{L}$,  thus
$$\phi^+(w_1)+\phi^-(w_1)+\phi^1(w_1)+\cdots + \phi^{d-2}(w_1), \quad w_1\in \F,$$
are linearly equivalent. By Abel's Theorem,  $\varphi$ is constant, so $d_{w_1} \varphi =0$ for all $w_1 \in \F$.  In particular, $d_{u_1}\varphi=0$.
Since  
 $d_{u_1}\varphi^i=0$ for all $i=1, \cdots, d-2$,  we have
$$d_{u_1}\varphi^++d_{u_1}\varphi^-=0,$$
thus $d_{u_1}\theta=d_{u_1}\varphi^+-d_{u_1}\varphi^-=2d_{u_1}\varphi^+\neq 0.$

\subsubsection{Proof of the nonzero property}\label{ss:nonzero}

For any algebraic curve $C$, the albanese variety denoted by $\opname{Alb}(C)$ is exactly the Jacobian variety $J(C)$, so we have $\opname{Alb}(\F)=J(\F)$. By the universal property of the Albanese variety, there is a unique holomorphic map (up to a translation) $\chi: J(\F)\to J(X)$ such that $\theta=\chi\circ \a_{\F}$, \ie, the following diagram commutes
\begin{equation*}
\xymatrix{\F\ar[d]_{\a_{\F}}\ar[r]^{\theta}& J(X)\\
J(\F)\ar[ur]_{\chi}}.
\end{equation*}
Since $\a_{\F*}: H_1(\F,\Z)\to H_1(J(\F), \Z)$ is an isomorphism, in order to show  $\theta_{*}=\chi_{*}\circ \a_{\F*}\neq 0$, we only need to show $\chi_{*}\neq 0$.

We can ignore the translation in our discussion, since the differential of a translation is zero and the translation induces a zero map on the homology by the following proposition.
\begin{proposition}\label{translation}
Let $M$ be a complex manifold and $T=\C^g/\Lambda$ be a complex torus with the lattice $\Lambda$ generated by $\omega_1, \cdots, \omega_{2g}$, and let $\varphi : M \to T$ be a morphism. If $\varphi$ is constant, \ie, it is a translation on $T$, then $\varphi_{*}=0 : H_k(M,\Z) \to H_k(T,\Z)$ for all $k\in \Z$.
\end{proposition}
\begin{proof}
 Since $\varphi$ is constant, there is a point $z_0\in T$ with $z_0=x_1\omega_1+\cdots+x_{2g}\omega_{2g}$, $x_i\in [0,1)$, $i=1, \cdots, 2g$, such that $\varphi=z_0$. Denote by $I=[0,1]$ and define
 \begin{equation*}
 \begin{aligned}
 \Phi : M \times I &\to T\\
 (m,t) &\mapsto tz_0,
 \end{aligned}
 \end{equation*}
where $tz_0=tx_1\omega_1+\cdots+tx_{2g}\omega_{2g}$ and $tx_i\in [0,1)$. Clearly, by the definition of continuity in term of limit, $\Phi$ is continuous, so $\Phi$ is a homotopy. Since $\Phi(m,0)=0$ and $\Phi(m,1)=\varphi(m)$, thus $\varphi$ is homotopic to $0$, so $\varphi_{*}=0$.
\end{proof}

Denote by $J(X)=\C^{g(X)}/{\Lambda_{X}}$ and $J(\F)=\C^{g(\F)}/{\Lambda_{\F}}$, where $g(X), g(\F)$ are the genus and $\Lambda_X, \Lambda_{\F}$ are  lattices of $\C^{g(X)},\C^{g(\F)}$ respectively.
Then $\Lambda_{X}=H_1(J(X), \Z)=H_1(X, \Z)$ and $\Lambda_{\F}=H_1(J(\F),\Z)=H_1(\F, \Z)$.
Since $\chi$ is a holomorphic map between two complex tori, it lifts to an affine linear transformation
\begin{equation*}
\begin{aligned}
\tilde\chi :& \C^{g({\F})} \to \C^{g({X})}\\
& z\mapsto Az+b,
\end{aligned}
\end{equation*}
where $A$ is a constant $g(X)$ by $g(\F)$ matrix and $b$ is a constant $g(X)$-column vector. We obtain that $d\chi=A$ is constant and then $d\theta=(d\chi)\circ(d \a_{\F})=A(d \a_{\F})$.
Since $d_{u_1}\theta\neq 0$,  we get $A\neq 0$.
Note that the induced map $\chi_{*}$ of $\chi$ on homology maps $\Lambda_{\F}$ to $\Lambda_{X}$ and also preserves the zero vector, thus it should be $A$, \ie, $\chi_{*}=A \neq 0$.

\subsection{Proof of Lemma \ref{lem-top-2}}\label{s:stable}
In this section, $S$ is of dimension $2n$.
Consider the monodromy representation
$$ \rho_S : \pi_1(\L^{\sm}_S, r_0) \to \opname{Aut}\left(H_{2n-1}^{\van}(X, \Z), <,>\right),$$
where $<,>$ denotes the intersection pair.  More generally, we prove the following observation about the stable  $\Z$-submodules and Lemma \ref{lem-top-2} follows from this theorem for $n=1$.
\begin{theorem}\label{stable}
The stable  $\Z$-submodules of $H_{2n-1}^{\van}(X,\Z)$ under the monodromy action $\rho_S$
are $$mH_{2n-1}^{\van}(X,\Z), \qquad m\in \Z.$$
\end{theorem}

\begin{proof}
Let $V$ be a nonzero stable $\Z$-submodule of $H_{2n-1}^{\van}(X, \Z)$ under the monodromy action $\rho$, then we only need to show $V=mH_{2n-1}^{\van}(X, \Z)$ for some nonzero integer $m$.

Denote by $\L^{\sg} :=\L-\L^{\sm}$ parametrizing all singular hyperplane sections of $X$. By the Riemann--Hurwitz formula, $\deg \L^{\sg}$ is even and then write $\L^{\sg} :=\{t_1, \cdots, t_{2\ell}\}$. Let $\delta_i, \gamma_i$ be the elementary vanishing cycles corresponding to $t_{2i-1}, t_{2i}$ seperately for $i=1, 2,  \cdots, \ell$. 
We have \cite[Lemma 2.26]{Voisin2}
$$H_{2n-1}^{\van}(X, \Z)=\bigoplus\limits_{i=1}^\ell\left( \Z\delta_{i}\oplus \Z\gamma_{i}\right).$$

Since $2n-1=odd$, the intersection form $<,>$ is skewsymmetric, thus $<\delta_i, \delta_i>=0=<\gamma_i, \gamma_i>$ for all $i=1, \cdots, \ell$.  And since $<,>$ is unimodular, by rearranging $\{\delta_i, \gamma_i\}_{i=1}^\ell$, we can set $<\delta_i, \gamma_j>=\delta_{ij}$ and $<\delta_i, \delta_j>=0=<\gamma_i, \gamma_j>$.

For $k=1, \cdots, 2\ell$, let $\Delta_k$ be a small disk around $t_k$ and pick a point $r_k$ on the boundary of $\Delta_k$.
 For $i\in \{1, \cdots, \ell\}$, draw simple paths $\xi_i, \zeta_i$ in $\L^{\sm}_S$ connecting $r_0$ with $r_{2i-1}, r_{2i}$, respectively. Let $\tilde\xi_i$ ($\tilde\zeta_i$) be the loop in $\L^{\sm}_S$ based at $r_0$ which is equal to $\xi_i$ ($\zeta_i$) until $r_{2i-1}$ ($r_{2i}$), winds around the disk $\Delta_{2i-1}$ ($\Delta_{2i}$) once in the positive direction, and then returns to $r_0$ via $\xi_{2i-1}^{-1}$ ($\xi_{2i}$).

Pick a nonzero element $\alpha=\sum_{i=1}^\ell(a_i\delta_i+b_i\gamma_i)\in V$ with $a_i, b_i\in \Z$.

For $i\in \{1, \cdots, \ell\}$,  applying the Picard-Lefschetz formula, we have
$$ \rho(\tilde\zeta_i)(\alpha)=\alpha\pm <\alpha, \gamma_i>\gamma_i=\alpha \pm a_i\gamma_i, $$
then $a_i\gamma_i =\pm\left(\rho(\tilde\zeta_i)(\alpha) -\alpha\right) \in V$. Applying  the Picard-Lefschetz formula again, we have
  $$\rho(\tilde\xi_i)(a_i\gamma_i)=a_i\gamma_i\mp a_i\delta_i, $$
then $a_i\delta_i =\mp\left(\rho(\tilde\xi_i)(a_i\gamma_i) -a_i\gamma_i\right) \in V$. Therefore, $a_i H_{2n-1}^{\van}(X,\Z) \subseteq V$.
Similarly,  we have $b_i H_{2n-1}^{\van}(X,\Z) \subseteq V$.

Denote by $d_\alpha=\gcd(a_1,b_1,\cdots, a_\ell,b_\ell)\geq 1$, then there are integers $u_1, v_1, \cdots, u_\ell, v_\ell$ such that $d_\alpha=\sum\limits_{i=1}^\ell(u_ia_i+v_ib_i)$. For all $k=1, \cdots, \ell$, we have
\begin{equation*}
\begin{aligned}
d_\alpha\delta_k=\sum\limits_{i=1}^g\left(u_ia_i\delta_k+v_ib_i\delta_k\right) \in V,\\
d_\alpha\gamma_k=\sum\limits_{i=1}^g\left(u_ia_i\gamma_k+v_ib_i\gamma_k\right) \in V.
\end{aligned}
\end{equation*}
Therefore, $d_\alpha H_{2n-1}^{\van}(X,\Z)\subseteq V$.
If $V\neq d_\alpha H_{2n-1}^{\van}(X,\Z)$, there is
 $$\beta=\sum\limits_{i=1}^g(e_i\delta_i+f_i\gamma_i) \in V\setminus d_\alpha H_{2n-1}^{\van}(X,\Z), \qquad e_i, f_i \in \Z.$$
 Denote by $d_\beta=\gcd(e_1, f_1, \cdots, e_g, f_g)$, then $d_\alpha\nmid d_\beta$ because $\beta \not\in d_\alpha H_{2n-1}^{\van}(X,\Z)$. Applying the above calculation to $\beta$ gives us
 $d_\beta H_{2n-1}^{\van}(X,\Z) \subseteq V$, thus $\gamma:=d_\alpha\delta_1+d_\beta \gamma_1 \in V$ and $d_\gamma =\gcd (d_\alpha, d_\beta)<d_\alpha$. Again, by the  above calculation, we obtain $d_\gamma H_{2n-1}^{\van}(X,\Z) \subseteq V$.
 Therefore, we have
 $$d_\alpha H_{2n-1}^{\van}(X,\Z) \subsetneq d_\gamma H_{2n-1}^{\van}(X,\Z) \subseteq V\subseteq H_{2n-1}^{\van}(X,\Z) \qquad \text{with}\quad d_\alpha > d_\gamma \geq 1 \, \text{ and }\, d_\gamma | d_\alpha.$$
 Replacing $\alpha$ by $\gamma$, continue to do the above analysis. Since each process will give us a smaller positive integer, by finite many processes, we will get $V= mH_{2n-1}^{\van}(X,\Z)$ for some positive integer $m$, which is desired.
\end{proof}


\section{Jacobi-type Inversion Theorem}\label{sec-Jacobi}
As an application of our construction, we prove the Jacobi-type Inversion theorem \ref{thm-Jacobi}. Let $S$ is of dimension $2$.
By Lemma \ref{lem-top-1}, we can fix the deformation space $\F$ for all elementary vanishing cycles.
By the analysis in Section \ref{s:restate}, we have the following commutative diagrams
  \begin{equation*}
\xymatrix{&\left(\tilde{X}\times_{\tilde{\L}}\tilde{X}\setminus\Delta, (x_i^+, x_i^-)\right)\ar[d]_{\opname{q}} \ar[r]^>>>>>{\xi_i} &X^+\times X^- \ar[r]^{\a_X}& J(X)\\
(\F, s_i)\ar[r]^{\kappa_i}\ar@{.>}[ru]^{\eta_i}\ar@{.>}[rrru]|-{\theta_i}&(\tilde{\L}, t_i)}
\end{equation*}
and
 \begin{equation*}
 \xymatrix{&(\tilde X, x_i^{\pm})\ar[r]^-{\tilde\sigma_i}\ar[d]
& (X, x_i)\ar[r]^{\a_X}& J(X)\\
 (\F, s_i) \ar[r]\ar@{-->}[ru]^{f_i^{\pm}}\ar@{.>}[rru]|>>>>>>>{\phi_i^{\pm}}\ar@{.>}[rrru]|>>>>>>>>>>>{\varphi_i^{\pm}}&(\widetilde\L, t_i)},
 \end{equation*}
where $i=1, \cdots, d-2$.  Clearly, $\theta_i=\varphi_i^+ - \varphi_i^-$ and
$$d_{s_i}\theta_i=d_{s_i}\varphi_i^+-d_{s_i}\varphi_i^-=2d_{s_i}\varphi_i^+\neq 0.$$

Denote by $H^0(X, \Omega)=\bigoplus\limits_{k=1}^g \C\omega_k$.
We have the following holomorphic map:
\begin{equation*}
\begin{aligned}
\vartheta:=\sum\limits_{i=1}^{e}\theta_{i} : \quad \prod_{i=1}^{e} \F &\to J(X)\cong \C^g/\Lambda\\
(r_i)_{i=1}^e &\mapsto  \sum\limits_{i=1}^{e} \left(\varphi_i^+(r_i)-\varphi_i^-(r_i)\right)&=\left(\begin{array}{ccc}
 \sum\limits_{i=1}^e\int_{\phi_{i}^-(r_i)}^{\phi_{i}^+(r_i)}\omega_1&\cdots& \sum\limits_{i=1}^e\int_{\phi_{i}^-(r_i)}^{\phi_{i}^+(r_i)}\omega_g\end{array}\right).
\end{aligned}
\end{equation*}
Note that $\phi_{i}^+(s_i)=\phi_{i}^-(s_i)=x_i$ is the ramification point, then $\vartheta(s)=0$ and the Jacobian of $\vartheta$ at  $s:=(s_i)_{i=1}^e$ is
\begin{equation*}
\begin{aligned}
J(\vartheta)|_{s}=\left(\begin{array}{ccccc}
 \frac{\omega_1(x_1)}{dz_1}d_{s_1}\theta_{1}& \cdots& \frac{\omega_1(x_i)}{dz_i}d_{s_i}\theta_{i}&\cdots& \frac{\omega_1(x_e)}{dz_e}d_{s_e}\theta_{e}\\
 \vdots & &\vdots&& \cdots\\
  \frac{\omega_g(x_1)}{dz_1}d_{s_1}\theta_{1}& \cdots& \frac{\omega_g(x_i)}{dz_i}d_{s_i}\theta_{i}&\cdots& \frac{\omega_g(x_e)}{dz_e}d_{s_e}\theta_{e}
 \end{array}\right)_{g\times e},
\end{aligned}
\end{equation*}
where $z_i$ denotes the local coordinate around $x_i$ and $\omega_j=f_j(z_i)dz_i$ for some holomorphic map $f_j(z_i)$ locally around $x_i$. In fact, $  \frac{\omega_j(x_i)}{dz_i}=f_j(0)$.

Since $d_{s_i}\theta_{i}\neq 0$,  the rank of $J(\vartheta)|_s$ is the same as the following matrix
\begin{equation*}
\begin{aligned}
M(\vartheta, s)=\left(\begin{array}{ccccc}
 \frac{\omega_1(x_1)}{dz_1}& \cdots& \frac{\omega_1(x_i)}{dz_i}&\cdots& \frac{\omega_1(x_e)}{dz_e}\\
 \vdots & &\vdots&& \cdots\\
  \frac{\omega_g(x_1)}{dz_1}& \cdots& \frac{\omega_g(x_i)}{dz_i}&\cdots& \frac{\omega_g(x_e)}{dz_e}
 \end{array}\right)_{g\times e}.
\end{aligned}
\end{equation*}
If $\opname{rank} J(\vartheta)|_{s} = \opname{rank} M(\vartheta, s) <g$, then the row vectors of $M(\vartheta, s)$ are $\C$-linear dependent, \ie, there are $a_1, \cdots, a_g \in \C$ which are not all 0, such that
\begin{equation*}
\begin{aligned}
\frac{\left(\sum\limits_{j=1}^ga_j\omega_j\right)(x_i)}{dz_i}=a_1 \frac{\omega_1(x_i)}{dz_i}+\cdots +a_g\frac{\omega_g(x_i)}{dz_i}=0 ,  i=1, \cdots, e
\end{aligned}
\end{equation*}
We obtain $x_i \in Z(\sum\limits_{j=1}^ga_j\omega_j)$ for all $i=1, \cdots, e$, \ie, $\opname{div}\left(\sum\limits_{j=1}^ga_j\omega_j\right) \geq \sum\limits_{i=1}^e x_i$, then
$$0\neq \sum\limits_{j=1}^ga_j\omega_j \in H^0(X, \Omega(- \sum\limits_{i=1}^e x_i)).$$
 However, $\deg (K-\sum\limits_{i=1}^e x_i))=2g-2-e=2g-2-2(d+g-1)=-2d<0$, which means $H^0(X, \Omega(- \sum\limits_{i=1}^e x_i))=0$.
We get a contradiction. Therefore, $\opname{rank} J(\vartheta)|_s=g$. 
Assume that the first $g$ column vectors in $ J(\vartheta)|_s$ are $\C$-linear independent, then $J(\widetilde\theta)|_s$ is invertible, where
\begin{equation*}
\begin{aligned}
\widetilde\theta:=\sum\limits_{i=1}^g\theta_{i} : \quad \prod_{i=1}^g \F &\to J(X)\cong \C^g/\Lambda\\
(r_i)_{i=1}^g &\mapsto  \sum\limits_{i=1}^g \theta_{i}(r_i)&=\left(\begin{array}{ccc}
 \sum\limits_{i=1}^g\int_{\phi_{i}^-(r_i)}^{\phi_{i}^+(r_i)}\omega_1&\cdots& \sum\limits_{i=1}^g\int_{\phi_{i}^-(r_i)}^{\phi_{i}^+(r_i)}\omega_g\end{array}\right).
\end{aligned}
\end{equation*}
By the following proposition, we know $\widetilde\theta$ is surjective, which proves the Jacobi-type Inversion Theorem \ref{thm-Jacobi}.


\begin{proposition}\label{prop}
Let $f: M \to N$ be a holomorphic map between two complex manifold, where $M$ is compact and $N$ is connected. If $\dim M \geq \dim N$ and $\opname{rank} J(f)|_x = \dim N$ for some $x\in M$, then $f$ is surjective.
\end{proposition}
\begin{proof}
Since $f$ is continuous and $M$ is compact, by Remmert Proper Mapping theorem, $f$ is proper, so $f(M)$ is an analytic subvariety of $N$.

Note that an analytic variety is irreducible if and only if the smooth part is connected \cite[p. 21]{GH-Principles}, thus $N$ is irreducible because $N$ is a connected complex manifold.

Since $\opname{rank} J(f)|_x = \dim N$, $f$ is a submersion at $x$, then there is an open neighborhood $U$ of $x$ in $M$ and an open neighborhood $V$ of $f(x)$ in $N$ such that $f(U)=V$ and $f|_U$ is a projection. Note that $V=f(U) \subseteq f(M)$, then $N=f(M)\cup V^c$. Since $N$ is irreducible and $f(M)\neq \emptyset$, we obtain that $f(M)=N$.
\end{proof}

\section{Geometry of the deformation space}\label{sec-geo}
 In this section, we will give a geometric description for the deformation space $\overline{\U}$ of an elementary vanishing cycle and then prove the Lefschetz $(1,1)$--theorem using our construction.

Let $S$ be a complex projective surface of degree $d$ in $\P^N$. Denote by $\O:=(\P^N)^\vee$ which parametrizes all hyperplane sections of $S$.
In order to construct a generic net, we consider  the quotient bundle  \begin{equation*}\label{quot-bdle}
\mathcal{Q}:=\frac{{\O\times V}}{\mathcal{T}}
\end{equation*}
over $\O$, where $\O=\P(V)$ for some complex vector space $V\cong \C^{N+1}$ and
$\mathcal{T}$ denotes
the tautological line bundle 
over $\O$  whose fiber at any point is the line in $V$ represented by the point. Clearly, the fiber of the corresponding projective bundle $\P\mathcal{Q}$ over each $t\in \O$ parametrizes all hyperplane sections of $S_t$.
Pick a Lefschetz pencil $\L$ in $\O$. Denote the pullback bundle over $\L$ by $\mathcal{Q}_{\L}=\L\otimes_{\O} \mathcal{Q}$.
Note that the first Chern class of $\mathcal{Q}$ is $1$,
We have the following conclusion.
\begin{lemma}\label{lem-bdle}
 The notations are described above, we have
 \begin{equation*}\label{quot-bdle}
\mathcal{Q}_{\L}\cong \mathcal{O}_{\L}(1)\oplus \mathcal{O}_{\L}^{\oplus N-1}.
\end{equation*}
\end{lemma}
\begin{proof}
By Grothendieck's classification theorem of Vector bundles over $\P^1$ \cite{Grothendieck}, there exist integers $n_1, \cdots, n_k>0$ and $m_1, \cdots, m_k$ with $m_1>\cdots >m_k$ and  $n_1+\cdots +n_k=N$ such that
$$\mathcal{Q}_{\L}=\mathcal{O}_{\L}(m_1)^{\oplus n_1}\oplus\cdots\oplus\mathscr{O}_{\L}(m_k)^{\oplus n_k}.$$
Since $\mathcal{Q}$ is a quotient of a trivial bundle, it can be generated by nonzero global sections, so is $\mathcal{Q}_{\L}$. Note that  a line bundle  of negative degree over a smooth projective complex curve  has no nonzero global sections,  we have $$m_1>\cdots>m_k\geq 0.$$
By the Whitney sum formula, the first chern class of $\mathcal{Q}_{\L}$ is
$$c_1(\mathcal{Q}_{\L})=c_1(\mathcal{Q})=c_1(\O\times V)-c_1(\mathcal{T})=0-(-1)=1,$$
then $n_1m_1+\cdots +n_km_k=1$, which gives us
 $k=2$, $n_1=m_1=1$ and $m_2=0$, $n_2=N-1$.
Therefore,
$$\mathcal{Q}_{\L}\cong \mathcal{O}_{\L}(1)\oplus \mathcal{O}_{\L}^{\oplus (N-1)}.$$
\end{proof}

By Lemma \ref{lem-bdle}, we can pick a general projective subbundle $\PP$  of rank $1$ over $\L$ in $\P\mathcal{Q}_\L$  such that
\begin{equation*}\label{quot-bdle}
\PP\cong \P(\mathcal{O}_{\L}(1)\oplus \mathcal{O}_\L^{\oplus 2}).
\end{equation*}
and the fiber
\begin{equation*}\label{quot-bdle}
\PP_t\cong \P(\mathcal{O}_{\L}(1)\oplus \mathcal{O}_\L)
\end{equation*}
over each $t\in \L^{\sm}$ is a Lefschetz pencil for $S_t$.
Denote by $\PP^{\sm}$ the open subset of $\PP$ corresponding to all smooth hyperplane sections of $S_t$ for all $t\in \L^{\sm}$, i.e., $S_{t,r}=\{P_0, P_1, \cdots, P_{d-1}\}$ is a set of $d$ distinct points for all $(t, r)\in \PP^{\sm}$.
We have the following commutative diagram
$$\xymatrix{
S_{\PP^{\sm}}:=S_{\L^{\sm}}{\times_{\L^{\sm}}} \PP^{\sm} \ar[d]^{\sigma^{\sm}} \ar[r]  & S_{\L^{\sm}} \ar[d]_{\pi^{\sm}}\ar[r]         &S \\
\PP^{\sm}  \ar[r]^{\rho^{\sm}}  & \L^{\sm}     &       }.$$
where $S_{\L^{\sm}}=\left\{(x, t)\in S\times \L^{\sm}\,|\, t(x)=0\right\}$  and $S_{\PP^{\sm}}=\left\{(x, t,r)\in S\times \PP^{\sm}\,|\, t(x)=r(x)=0\right\}$ denote the incidence varieties.

We  call such $\PP$ as a generic net which means the complement $\Sigma:=\PP\setminus \PP^{\sm}$
is an algebraic curve with only finite many ordinary nodes and ordinary cusps as its singularities. More generally, we have the following conclusion.
\begin{proposition}\label{prop-net}
Let $S$ be a smooth irreducible complex projective surface. By embedding $S$ into a sufficient large projective space $\P^N$, any general 2-plane $\PP$ in $\O:=(\P^N)^\vee$ forms a generic net which means that $\PP\setminus \PP^{\sm}$ is an algebraic curve with only finite many ordinary nodes and ordinary cusps as its singularities. Equivalently,
there are only finitely many singular hyperplane sections that either have at most two ordinary double points as singularities or have only one ordinary cusp point as singularity.
\end{proposition}

First of all, we will prove the following lemma for later usage.

\begin{lemma}\label{lem-Bertini}
 Let $Y$ be an irreducible hypersurface in $\P^N$ with $N\geq 3$ and $H$ be a general hyperplane in $\P^N$, then $Y\cap H$ is irreducible.
\end{lemma}
\begin{proof}
Let $Y=Z(F)$, where $F\in \C[X_0, \cdots, X_N]$ is an irreducible homogeneous polynomial, then the coordinate ring $\C[Y]=\C[X_0, \cdots, X_N]/(F)$.  Denote by $\C(Y)$ the function field of $Y$, \ie, the fraction field of $\C[Y]$.
Since $\dim Y = \opname{tr.deg}_\C \C(Y)=N-1\geq 2$, we can choose 3 linear homogeneous polynomials $L_0, L_1, L_2 \in \C[X_0, \cdots, X_N]$ such that $f_1=\frac{L_1}{L_0}|_Y$ and $f_2=\frac{L_2}{L_0}|_Y$ are algebraically independent in $\C(Y)$.

 Consider the incidence variety
\begin{align*}
\mathfrak{Y}&=\left\{(x,H)\in \P^N\times(\P^N)^\vee \,|\, x\in H\cap Y\right\}\subseteq Y\times (\P^N)^\vee\\
&=\left\{\left([X_0, \cdots, X_N], [a_0, \cdots, a_N]\right)
\,\big |\, F(X_0, \cdots, X_N)=0, a_0X_0+\cdots + a_NX_N=0\right\},
\end{align*}
which is a $\P^{N-1}$-bundle over $Y$. Since $Y$ is irreducible, so is $\mathfrak{Y}$.
Denote by $\opname{pr}_2: \mathfrak{Y} \twoheadrightarrow (\P^N)^\vee$  the second projection. Consider the first Bertini theorem \cite[p. 139]{Basic}.

\medskip

\noindent{\bf First Bertini Theorem.}\textit{
Let $M$ and $N$ be irreducible varieties defined over a field of characteristic $0$, and $f: M\to N$ a regular map such that $f(M)$ is dense in $N$. Suppose that $M$ remains irreducible over the algebraic closure $\overline{k(N)}$ of $k(N)$. Then there exists an open dense set $U\subseteq N$ such that all the fibres $f^{-1}(y)
$ over $y\in U$ are irreducible.}

\medskip

Generally, a variety $M$ is irreducible over $\overline{k(N)}$ if and only if the map $f^*: k(N) \hookrightarrow k(M)$ embeds $k(N)$ as an algebraically closed subfield of $k(M)$.
Therefore, there is an open dense subset $U\subseteq (\P^N)^\vee$ such that all the fibres $Y\cap H$ over $H\in (\P^N)^\vee $ are irreducible.
\end{proof}

\noindent\textit{Proof of Proposition \ref{prop-net}.}
Pick a general 2-plane $\mathcal{P}\cong \P^2$ in  $\O$, we need to show $\mathcal{P}\setminus\PP^{\sm}=\PP\cap S^\vee$ has only finitely many ordinary nodes and ordinary cusps as its singularities. Let
$P_{N-3}\cong \P^{N-3}$ be the general $(N-3)$-plane  in $\P^N$ corresponding to $\mathcal{P}$. Since a general 2-plane in $\O$ corresponds to a general
$(N-3)$-plane in $\P^N$,  then $P_{N-3}\cap S=\emptyset$ and
$$\mathcal{P}=\{H \in \O \,|\, P_{N-3}\subseteq H\} \cong \P^2.$$
Consider the quotient space
$$\P^N/P_{N-3}=\left\{H\subseteq \P^N \,|\,  P_{N-3}\subseteq H\cong \P^{N-2} \right\}\cong \P^2.$$
We have a projection with center $P_{N-3}$
\begin{align*}
\pi:  \P^N\setminus P_{N-3}&\to \P^N/P_{N-3}\\
x&\mapsto  \pi(x)=\text{\,the } (N-2)\text{-plane in } \P^N \text{ containing } P_{N-3} \text{ and } x.
\end{align*}
Clearly, $\left(\P^N/P_{N-3}\right)^\vee\cong \mathcal{P}$, thus we need to show $\mathcal{P}\cap S^\vee\cong \left(\P^N/P_{N-3}\right)^\vee\cap S^\vee $ has only finitely many ordinary nodes and ordinary cusps as its singularities.

Since $P_{N-3}$ is general and $S^\vee$ is irreducible, by Lemma \ref{lem-Bertini}, $\left(\P^N/P_{N-3}\right)^\vee\cap S^\vee$ is irreducible. Note that $\pi(S)^\vee$ is an irreducible component of $\left(\P/P_{N-3}\right)^\vee\cap S^\vee$ \cite[Theorem 1.21 on p. 11]{ProjDual},
then $\pi(S)^\vee=\left(\P^N/P_{N-3}\right)^\vee\cap S^\vee$. Therefore, we only need to show $\pi(S)^\vee$ also has only finitely many ordinary nodes and ordinary cusps as its singularities, \ie, the dual curve of a general projection of a smooth curve on a 2-plane has only finite many ordinary nodes and ordinary cusps as its singularities, because
the projection of a smooth irreducible complex projective curve on a
general plane is generic.
$\Box$

Next,  Nori¡¯s famous Connectivity
Theorem \cite[Corollary 4.4 on p. 364]{Nori} gives us the following isomorphisms
\begin{align*}
H_{\prim}^{2}(S, \Q)&\cong H^1(\L^{\sm}, R^{1}_{\van}\pi^{\sm}_*\Q),\label{iso-1}\\
    H_{\prim}^{1}(S_t, \Q)&\cong H^1(\PP_t^{\sm}, R^{0}_{\van}\sigma^{\sm}_{t*}\Q),\qquad t\in \L^{\sm},
\end{align*}
which lead us to consider $H_{\van}^{1}(S_t, \Z)$ and the local system $R_{\van}^0\sigma_*^{\sm}(\Z)$ over $\mathcal{P}^{\sm}$, which are fiberwise given by $$H_{\van}^0(S_{t,r}, \Z)=\bigoplus\limits_{i=1}^{d-1} \Z(P_i-P_0)$$ with $S_{t,r}=\{P_0, P_1, \cdots, P_{d-1}\}$. Via the flat Gauss-Manin connection, the total space naturally forms an infinite-sheeted covering space
$$|R_{\van}^0\sigma_*(\Z)|\to \mathcal{P}^{\sm}.$$
Fix an elementary vanishing cycle $\alpha:=P_1-P_0$, there is a unique component $\mathbb{U}$ in $|R_{\van}^0\sigma_*^{\sm}(\Z)|$ corresponding to the stabilizer of $\alpha$ in $\pi_1(\mathcal{P}^{\sm})$. For any $t\in \L^{\sm}$, by composing the topological Abel--Jacobi mapping  \cite{TopAJ} with the projection $J_{\prim}(S_t) \to J_{\van}(S_t)$, we consider
$$\a_t: \mathbb{U}_t \to J_{\van}(S_t).$$
We will use the fiber $\mathbb{U}_t$ to construct a parametrization of $J_{\van}(S_t)$ for all $t\in \L^{\sm}$.

Following the idea in \cite{HodgeLocus}, we construct a canonical completion $\overline{\mathbb{U}}$ of $\mathbb{U}$ to a normal analytic variety and with a proper morphism
$$\overline{\mathbb{U}}\xrightarrow{\overline\sigma} \mathcal{P}$$
which extends $\mathbb{U}\xrightarrow{\sigma} \mathcal{P}^{\sm}$. 
Before we study the geometry of this completion $\overline{\U}$, we have to recall the cyclic quotient singularities.
First of all, we recall a well known conclusion which will be used later.
\begin{lemma}\label{lem-sing}
Let $X$ be an affine variety over a field $k$ and $G$ be a finite group acting on $X$ by algebraic automorphisms over $k$. Denote by $k[X]$  and $k[X/G]$ the coordinate rings of $X$ and $X/G$, respectively. Then
\begin{enumerate}
\item $k[X/G]=k[X]^G=k[y_1,\cdots, y_m]/I$, where $y_1,\cdots, y_m$ are algebraically independent generators and $I$ the ideal of relations among  them.
\item $X/G =\opname{Spec} k[y_1,\cdots, y_m]/I = Z(I)$.
\end{enumerate}
\end{lemma}
Let $\xi_n$ be a primitive $n$-th root of unity and $k$ a positive integer with $g.c.d(n,k)=1$ and $k<n$.
Consider the cyclic group $C_n=<\xi_n>$ acting on $\C^2$ by $\xi_n \dot (x,y)=( \xi_n x, \xi_n^ky)$.
We call the quotient analytic space $X_{n,k}=\C^2/C_n$ has  \emph{a cyclic quotient singularity} of type $\frac{1}{n}(1,k)$ and
\begin{itemize}
\item $X_{n,k} \cong X_{n',k'}$ if and only if $n=n'$ and either $k=k'$ or $kk' \equiv 1 (\opname{mod} n)$;
\item The exceptional divisor on the minimal resolution $\tilde X_{n,k}$ of $X_{n, k}$ is a Hirzebruch--Jung string, \ie, a connected union $E=\bigcup\limits_{i=1}^{m}E_i$  of smooth rational curves $E_1, \cdots, E_m$  with
\begin{enumerate}
\item[(1)] $E_i\dot E_i = -b_i\leq -2$ given the continued fraction $\frac{n}{k}=[b_1, \cdots, b_m]:=b_1-\frac{1}{b_2-\frac{1}{\cdots - \frac{1}{b_m}}}$.
\item[(2)] $E_iE_{i+1}=1$.
\item[(3)] $E_iE_j=0$ if $|i-j|\geq 2$.
\end{enumerate}
\item $\pi_1(X_{n,k}-\{0\})=C_n$ because $\C^2-\{0\}$ is simply connected and the action on $\C^2-\{0\}$ is free and properly discontinuous (because $C_n$ is of finite order).
\end{itemize}

We only need the following two cases which are described in detail here.
\begin{itemize}
    \item \textit{\bf The node case.}\label{s2} 
    The action of $C_2$ on $\C^2$ is given by $(-1)\cdot (x,y)=( -x, - y)$.
Note that $\C[\C^2]=\C[x,y]$,  by Lemma \ref{lem-sing},  we have
$$\C[X_{2,1}]=\C[x,y]^{C_2}=\C[x^2,xy, y^2]=\C[u,v,w]/<uw-v^2>,$$
so $X_{2,1}=Z(uw-v^2) \subseteq \C^3$, which has only one singularity $0=(0,0,0)$ and $\pi_1(X_{2,1}-\{0\})=\Z/2$.
Considering the continued fraction $\frac{2}{1}=2$, we have $m=1$ and $b_1=2$, \ie, the exceptional divisor $E$ is a smooth irreducible rational curve with $E\cdot E=-2$.
\end{itemize}

\begin{itemize}
    \item\label{s3} {\bf The cusp case.} 
    The action of $C_3$ on $\C^2$ is given by $\xi\cdot (x,y)=( \xi x, \xi^2 y)$, where $\xi=e^{\frac{2\pi\sqrt{-1}}{3}}$.
Note that $\C[\C^2]=\C[x,y]$, similarly, by Lemma \ref{lem-sing}, we have
$$\C[X_{3,2}]=\C[x,y]^{C_3}=\C[x^3,xy, y^3]=\C[u,v,w]/<uw-v^3>,$$
so $X_{3,2}=Z(uw-v^3) \subseteq \C^3$, which has only one singularity $0=(0,0,0)$ and $\pi_1(X_{3,2}-\{0\})=\Z/3$.
Considering the continued fraction $\frac{3}{2}=2-\frac{1}{2}$, we have $m=2$ and $b_1=b_2=2$, \ie, the exceptional divisor $E=E_1\cup E_2$, where $E_1$ and $E_2$ are smooth irreducible rational curves with $E_i\cdot E_i=-2$ for $i=1,2$ and $E_1\cdot E_2=1$.
\end{itemize}

Denote by  $\Sigma^{\sg}$ the set of singularities in $\Sigma$ and $\Sigma^{\sm}:=\Sigma-\Sigma^{\sg}$.
Let $\widehat{\mathcal{P}}$ be the blowup of $\mathcal{P}$ along $\Sigma^{\sg}$ and $\widehat{\Sigma}$ denotes the proper transform of $\Sigma$ in $\widehat{\mathcal{P}}$. Denote by $\tilde{\mathcal{P}}$ the double cover of $\widehat{\mathcal{P}}$ branched along $\widehat{\Sigma}$. In fact, $\tilde{\mathcal{P}}$ is smooth and is the minimal resolution of the double cover of ${\mathcal{P}}$ branched along $\Sigma$. We have the following conclusion about the geometry of $\overline{\mathbb{U}}$.

\begin{theorem}
Let $S_{\mathcal{P}}=\{(x, t, r)\in S\times \mathcal{P}\,|\, t(x)=r(x)=0\}$ and $S_{\tilde{\mathcal{P}}}=\tilde{\mathcal{P}}\times_{\mathcal{P}}S_{{\mathcal{P}}}$, then we have
$$\overline{\mathbb{U}} \cong S_{\tilde{\mathcal{P}}}\times_{\tilde{\mathcal{P}}}S_{\tilde{\mathcal{P}}}\setminus\Delta.$$
\end{theorem}

\begin{proof}
Denote by $\PP'=\PP-\Sigma^{\sg}$ and consider the incidence variety
\begin{align*}
S_{\PP'}&=\left\{(x,t,r)\in S\times \PP' \,|\, t(x)=r(x)=0\right\}.
\end{align*}
Denoting by $\tilde{S_{\PP'}}$ the normalization  of $S_{\PP'}\times_{\PP'}\tilde{\PP'}$, we have the following commutative diagram
$$\xymatrix{\tilde{S_{\PP'}}\times_{\tilde{\PP'}}
\tilde{S_{\PP'}}\ar[d]\ar[r]\ar[dr]^{\pi'}&\tilde{S_{\PP'}}\ar[d]^{\tilde\pi}\ar[r]^g&S_{\PP'}\ar[r]\ar[d]^{\pi}_{d:1}&S\\
\tilde{S_{\PP'}}\ar[r]^{\tilde{\pi}}&\tilde{\PP'}\ar[r]^{2:1}_{\tau}&\PP'},$$
where  $\pi$ is a covering map of degree $d$ branched along $\Sigma^{\sm}$ and $\tilde\pi$ is an unbranched covering map of degree $d$, thus $\pi'$ is an unbranched covering map of degree $d^2$.
Note that the diagonal $\Delta: =\Delta\left(\tilde{S_{\PP'}}\times_{\tilde{\PP'}}\tilde{S_{\PP'}}\right) \cong \tilde{S_{\PP'}}$, then restricting $\pi'$ to $\tilde{S_{\PP'}}\times_{\tilde{\PP'}}\tilde{S_{\PP'}} \setminus \Delta$ denoted by
$$\pi'':  \tilde{S_{\PP'}}\times_{\tilde{\PP'}}\tilde{S_{\PP'}}\setminus \Delta \to \tilde{\PP'}$$
is  an unbranched covering map of degree $d^2-d$.

 Since $\Sigma^{\sm}$ is path-connected,  $\tilde{S_{\PP'}}\times_{\tilde{\PP'}}\tilde{S_{\PP'}}\setminus\Delta $ is path-connected.
Pick $(t_0,r_0)\in \Sigma^{\sm}$.
Denote by
\begin{itemize}
\item $\pi^{-1}(t_0,r_0)=\{ 2x_0, x_1, \cdots, x_{d-2}\}$, where $x_0$ is the unique node on $\pi^{-1}(t_0,r_0)$,
\item $\tilde{\pi}^{-1}(t_0,r_0)=\{x_0^+, x_0^-, x_1, \cdots, x_{d-2}\}$, where $\{x_0^+, x_0^-\}:=g^{-1}(x_0)$.
\end{itemize}
Pick an elementary vanishing cycle $[x_0^+-x_0^-]\in H_0(S_{t_0}, \Z)_{\van}$ corresponding to $x_0$.  We will see there is a finite covering space $\G$ of $\PP$ branched along $\Sigma$ and fixing  $[x_0^+-x_0^-]$.

%
Note that $\Xi_{0}:=(\pi'')^{-1}(t_0,r_0)$ is the set of all ordered pairs of distinct points in $\tilde{\pi}^{-1}(t_0,r_0)$, so $|\Xi_{0}|=d(d-1)$.
We have a monodromy representation associated to the covering map $\pi''$
\begin{align*}
\rho: \pi_1\left(\tilde{\PP'}, (t_0,r_0)\right) \to \opname{Aut}\left(\Xi_{0}\right)\cong S_{d(d-1)}.
\end{align*}
Let $\opname{Aut}\left(\Xi_0\right)_{(x_0^+,x_0^-)}$ be the stabilizer of the ordered pair $(x_0^+, x_0^-)$ in
 $\opname{Aut}\left(\Xi_0\right)$, 
 then $$\opname{Aut}\left(\Xi_0\right)_{(x_0^+,x_0^-)}\cong S_{d(d-1)-1}.$$
 Denote by $G:=\rho^{-1}\left(\opname{Aut}\left(\Xi_0\right)_{(x_0^+,x_0^-)}\right)$, which is a subgroup of
 $\pi_1\left(\tilde{\PP'}, (t_0,r_0)\right)$. Note that $\rho$ induces an injective group homomorphism
 $$\frac{\pi_1\left(\tilde{\PP'}, (t_0,r_0)\right)}{G} \hookrightarrow \frac{\opname{Aut}\left(\Xi_0\right)}{\opname{Aut}\left(\Xi_0\right)_{(x_0^+,x_0^-)}}\cong \frac{S_{d(d-1)}}{S_{d(d-1)-1}},$$
 so the group index
 $$\left[\pi_1(\tilde{\PP'}, t_0): G\right]\leq \frac{|S_{d(d-1)}|}{|S_{d(d-1)-1}|}=\frac{[d(d-1)]!}{[d(d-1)-1]!}=d(d-1).$$
Let $\sigma: (\G', s_0) \to \left(\tilde{\PP'}, (t_0,r_0)\right)$ be the path-connected unbranched covering space of $\tilde{\PP'}$ corresponding to $G$, then
$$\deg \sigma= \left[\pi_1\left(\tilde{\P}', (t_0,r_0)\right): G\right] \leq d(d-1).$$
Since
$\sigma_*\left(\pi_1(\G', s_0)\right)=G$ 
 is a subgroup of $\pi''_*\left(\pi_1\left(\tilde{S_{\PP'}}\times_{\tilde{\PP'}}\tilde{S_{\PP'}}\setminus \Delta , (x_0^+, x_0^-)\right)\right)$, by the lifting criterion, there exists a map $\eta: (\G', s_0)\to\left(\tilde{S_{\PP'}}\times_{\tilde{\PP'}}\tilde{S_{\PP'}}\setminus \Delta, (x_0^+, x_0^-)\right)$ such that the following diagram commutes
 \begin{equation*}
\xymatrix{&\left(\tilde{S_{\PP'}}\times_{\tilde{\PP'}}\tilde{S_{\PP'}}\setminus \Delta, (x_0^+, x_0^-)\right)\ar[d]^{\pi''}\\
(\G', s_0)\ar[r]_{\sigma}\ar@{.>}[ru]^{\eta}&\left(\tilde{\PP'}, (t_0,r_0)\right)}.
\end{equation*}

Since both $\G'$ and $\tilde{S_{\PP'}}\times_{\tilde{\PP'}}\tilde{S_{\PP'}}\setminus \Delta$ are path-connected and
$$\sigma_*\left(\pi_1(\G', t_0)\right)=G=\rho^{-1}\left(\opname{Aut}\left(P_{t_0}\right)_{(x_0^+,x_0^-)}\right)=\pi''_*\left(\pi_1\left(\tilde{S_{\PP'}}\times_{\tilde{\PP'}}\tilde{S_{\PP'}}\setminus \Delta, (x_0^+, x_0^-)\right)\right),$$
by the uniqueness of the classification of covering spaces, we know
$$\G'\cong\tilde{S_{\PP'}}\times_{\tilde{\PP'}}\tilde{S_{\PP'}}\setminus \Delta.$$

In order to complete the covering space $\G'$, we need to analyze the local property for each point on  $\Sigma^{\sg}$, where we have two types of singularities. 
\begin{itemize}
\item {\bf Node points: }
Let $P\in \Sigma^{\sg}$ be a node point, and pick a small open ball $U_P$ centered at $P$ in $\P$ with local coordinates $\{x,y\}$ such that $U_P\cap \Sigma= Z(xy)$.
Denote by $\tilde U_P$ the double cover of $U$ branched along $U_P\cap \Sigma$ on $\C^3$ with coordinates $\{x,y,z\}$, then $\tilde U_P=Z(xy-z^2) \subseteq \tilde \P$, which is exactly $X_{2,1}$. 
Hence, we have $\tilde U_{P}\cong \C^2/{\Z_2}$ and $\pi_1(\tilde U_{P}^\circ)=\Z_2$, where $\tilde U_{P}^\circ=\tilde U_{P}-\{P\}$. Note that $\Z_2$ has two subgroups: $\{0\}$ and $\Z_2$, then $\tilde U_{P}^\circ$ has two kinds of covering spaces: the trivial covering space $\tilde U_{P}^\circ$ and the universal covering space that is isomorphic to $\C^2-\{(0,0)\}$.  Since $\sigma^{-1}(\tilde U_{P}^\circ)\subseteq \G'$ is a covering space of $\tilde U_{P}^\circ$,  we get that $\sigma^{-1}(U_{P}^\circ)$ consists of several copies of $\tilde U_{P}^\circ$ and $\C^2-\{(0,0)\}$. If the copy is $\tilde U_{P}^\circ$, we add $P$ to $\G'$ such that it becomes a surface cyclic quotient singularity of order 2.  If the copy is $\C^2-\{(0,0)\}$, we add the origin to $\G'$ such that it becomes a smooth point.

\item {\bf Cusp points:}
Let $Q\in \Sigma^{\sg}$ be a cusp point, and pick a small open ball $U_Q$ centered at $Q$ in $\P$ with local coordinates $\{x,y\}$ such that $U_Q\cap \Sigma= Z(x^2-y^3)$.
Denote by $\tilde U_Q$ the double cover of $U_Q$ branched along $U_Q\cap \Sigma$ on $\C^3$ with coordinates $\{x,y,z\}$, then $\tilde U_Q=Z(x^2-y^3-z^2) \subseteq \tilde \P$. Since $x^2-y^3-z^2=(x+z)(x-z)-y^3$, by changing local coordinates, $\tilde U_Q$ is exactly $X_{3,2}$. 
Hence, we have $\tilde U_{Q}\cong \C^2/{\Z_3}$ and $\pi_1(\tilde U_{Q}^\circ)=\Z_3$, where $\tilde U_{Q}^\circ:=\tilde U_{Q}-\{Q\}$. Note that  $\Z_3$ has two subgroups: $\{0\}$ and $\Z_3$, then $\tilde U_{Q}^\circ$ has two kinds of covering space: the trivial covering space $\tilde U_{Q}^\circ$ and the universal covering space that is isomorphic to $\C^2-\{(0,0)\}$.  Since $\sigma^{-1}(\tilde U_{Q}^\circ)\subseteq \G'$ is a covering space of $\tilde U_{Q}^\circ$, we get that $\sigma^{-1}(\tilde U_{Q}^\circ)$ consists of several copies of $\tilde U_{Q}^\circ$ and $\C^2-\{(0,0)\}$. If the copy is $\tilde U_{Q}^\circ$, we add $Q$ to $\G'$ such that it becomes a surface cyclic quotient singularity of order 3. If the copy is $\C^2-\{(0,0)\}$, we add the origin to $\G'$ such that it becomes a smooth point.
\end{itemize}
After adding several points to $\G'$ in the above way, we get a projective surface $\G$ with only finite many of surface cyclic quotient singularities of order 2 or 3.
Let $\pi_0: \tilde\G \to \G$ be 
 the blowup of $\G$ at all singular points. We have $\overline\U\cong\tilde\G={S_{\tilde\PP}}\times_{\tilde{\PP}}{S_{\tilde\PP}}\setminus \Delta $.
\end{proof}

 By the type of singularities in previous discussion, we know that the exceptional divisor over each node point is an  smooth irreducible rational curve (\ie, $\P^1$) with self-intersection number $-2$ and the exceptional divisor over each cusp point is two transversal smooth irreducible rational curves (\ie, two transversal $\P^1$) with self-intersection number $-2$. 

The geometry of the deformation space $\overline{\mathbb{U}}$ will help us to search the algebraic 0--cycles on $S_t$, which will trace out the algebraic 1--cycles dual to Hodge classes in $S$ as $t$ moves over $\L$.
In fact, the closure of $\mu_\eta(\L^{\sm})$ in $\overline{\mathbb{U}^{(N)}}$ is the algebraic 1--cycle, which is desired.




\begin{thebibliography}{100}

\bibitem{DiffTop}
Bjorn Ian Dundas, \emph{Differential Topology}, New York, NY, USA:Springer-Verlag, 2009.

\bibitem{Ehresmann}
C. Ehresmann, \emph{Sur les espaces fibr\'es diff\'erentiables}, C. R. Acad. Sci. Paris, \textbf{224}(1947), 1611--1612.



\bibitem{RatInt}
Phillip A. Griffiths,
\emph{On the periods of certain rational integrals. {I}, {II}},
Annals of Mathematics,
\textbf{90}, (1969),
460--495 and 496--541.


\bibitem{GH-Principles}
P. Griffiths and J. Harris, \emph{Principles of Algebraic Geometry}, Wiley Classics Library, John Wiley \& Sons, Inc., New York, Chichester, and Brisbane,  (1994), 419--426. MR 1288523 (95d:14001).

\bibitem{Grothendieck}
A. Grothendieck,
\emph{Sur la classification des fibr\'es holomorphes sur la sph\'ere de {Riemann}},
Amer. J. Math.,
\textbf{79},
(1957),
121--138.

\bibitem{Lef} Solomon Lefschetz,
\emph{L'analysis situs et la g{\'e}om{\'e}trie alg{\'e}brique}, Collection de monographies sur la th{\'e}orie des fonctions, publi{\'e} sous la direction de m. {\'E}mile  Borel, Reprint of the 1924 edition. Paris, France: Gauthier-Villars, (1950),  vi + 154.

\bibitem{Lewis} James D. Lewis,
\emph{A survey of the {Hodge} conjecture, 2nd edition}, \textbf{10},
(1999).






\bibitem{HodgeTheory}
Hossein Movasati, \emph{A Course in Hodge Theory: with emphasis on multiple integrals}, 2019,
\verb"http://w3.impa.br/~hossein/myarticles/hodgetheory.pdf".

\bibitem{Nori}
M.V. Nori, \emph{Algebraic cycles and Hodge-theoretic connectivity}, Invent. Math. \textbf{111}(2), 349--373(1993)

\bibitem{Tube}
Christian Schnell, \emph{Primitive cohomology and the tube mapping}, Math.Z. \textbf{268}(2011), no. 3-4, 1069--1089. MR 2818744(2012h:14061)

\bibitem{HodgeLocus} Christian Schnell, \emph{The extended locus of {Hodge} classes}, To appear in \emph{Publications of the Research  Institute for Mathematical Sciences}, arXiv:1401.7303,
(2014).

\bibitem{Basic} Igor R. Shafarevich, \emph{{Basic} algebraic geometry {I}, 2nd edition}, Springer-Verlag,
Berlin Heidelberg New York.


\bibitem{ProjDual} Evgueni A. Tevelev,
\emph{{Projective} duality and homogeneous spaces}, Springer Verlag, Berlin, (2005), 4-12.

\bibitem{Voisin1}
Claire Voisin, \emph{Hodge Theory and Complex Algebraic Geometry I}, Cambridge University press, 2002.
\bibitem{Voisin2}
Claire Voisin, \emph{Hodge Theory and Complex Algebraic Geometry II}, Cambridge University press, 2003.

\bibitem{Zucker-NF} Steven Zucker,
\emph{Generalized intermediate {Jacobians} and the theorem on normal functions}, Inventiones Mathematicae, \textbf{33}, no. 3, (1976), 185--222.




\bibitem{Zucker-4folds} Steven Zucker, \emph{The {Hodge} conjecture for cubic fourfolds}, Composition Mathematica, \textbf{34}, no. 2, (1977), 199--209.


\bibitem{TopAJ} Xiaolei Zhao, \emph{Topological {Abel--Jacobi} mapping and {Jacobi} inversion}, {Ph.D.} thesis, University of Michigan, Ann Arbor, MI, USA, (2015),
 \href{https://drive.google.com/file/d/1ApqvZRimavIsT2yMEGHHby2ZNKJ5TMtW/view}{\texttt MichiganLink ETD}.

\end{thebibliography}
\end{document}